\documentclass[final,12pt]{amsart}

\usepackage[pages=all, color=black, position={current page.south}, placement=bottom, scale=1, opacity=1, vshift=5mm]{background}

\usepackage[margin=1in]{geometry}
\usepackage{amssymb}
\usepackage{amsmath}
\usepackage{amsmath}
\usepackage{amsthm}
\usepackage{amssymb}
\usepackage{float}
\usepackage{multirow}
\usepackage{tikz}
\usepackage{enumerate}
\usepackage[caption=false]{subfig}
\usepackage[utf8]{inputenc}
\usepackage{hyperref}
\usepackage{algorithm}
\usepackage[export]{adjustbox} 
\usepackage{oubraces}
\usepackage{stix}
\usepackage{stmaryrd}
\theoremstyle{plain}
\newtheorem{theorem}{Theorem}[section]

\newtheorem{lemma}[theorem]{Lemma}

\theoremstyle{definition}

\newtheorem{example}{Example}
\newtheorem{Remark}[theorem]{Remark}

\newcommand\bx{\boldsymbol{x}}
\newcommand\bbeta{{\boldsymbol{\beta}}}
\newcommand\calP{{\mathcal{P}}}
\newcommand\calF{\mathcal{F}}

\newcommand\bU{\mathbf{U}}
\newcommand\bV{\mathbf{V}}
\newcommand\bW{\mathbf{W}}

\newcommand\bE{\mathbf{E}}

\usepackage{graphicx, color}
\graphicspath{{fig/}}

\usepackage{algorithm, algpseudocode}
\usepackage{mathrsfs} 
\usepackage{lipsum}

\author{
  Abdolreza Amiri
  }
  \address{
  Abdolreza Amiri
  \thanks{
    Department of Mathematics and Statistics,
    University of Strathclyde, 26 Richmond Street, Glasgow G1 1XK, UK,
    {\tt{abdolreza.amiri@strath.ac.uk}}.
}
}

\author{
  Gabriel R.  Barrenechea
  }
  \address{
  Gabriel R.  Barrenechea
  \thanks{
    Department of Mathematics and Statistics,
    University of Strathclyde, 26 Richmond Street, Glasgow G1 1XK, UK,
    {\tt{gabriel.barrenechea@strath.ac.uk}}.
}
}

\author{
  Tristan Pryer
}
\address{
  Tristan Pryer
  \thanks{
    Department of Mathematical Sciences,
    University of Bath, Bath BA2 7AY, UK
    {\tt{tmp38@bath.ac.uk}}.
}}

\title[Preserving the eigenvalue range]{A finite element method preserving the eigenvalue range
	of symmetric tensor fields}

\date{\today}

\begin{document}

\maketitle

\begin{abstract}
  This paper presents a finite element method that
    preserves (at the degrees of freedom) the eigenvalue range of the solution of tensor-valued
  time-dependent convection--diffusion equations.  Starting from a high-order spatial
  baseline discretisation (in this case, the CIP stabilised finite
  element method), our approach formulates the fully
    discrete problem as a variational inequality posed on a closed
  convex set of tensor-valued functions that respect the same
  eigenvalue bounds at their degrees of freedom. The
  numerical realisation of the scheme relies on the
  definition of a projection that, at each node, performs the
  diagonalisation of the tensor and then truncates the eigenvalues
 to lie within the prescribed bounds. The temporal
  discretisation is carried out using the implicit Euler method, and
  unconditional stability and optimal-order error
  estimates are proven for this choice. Numerical experiments confirm
  the theoretical findings and illustrate the method's ability to
  maintain eigenvalue constraints while accurately approximating
  solutions in the convection-dominated regime.
\end{abstract}

\noindent\underline{Keywords :} Tensor-valued PDEs, eigenvalue range, continuous interior penalty, finite element method,  convection-diffusion equations,  error estimates

\section{Introduction}
\label{sec:intro}

\subsection*{Motivation.}
Tensor-valued partial differential equations (PDEs) arise in many
settings where the primary unknown is more naturally represented by a
symmetric tensor field than by a scalar or a vector.  Examples include
elasticity \cite{Hughes2012}, non-Newtonian fluid mechanics
\cite{Owens2002}, diffusion tensor imaging \cite{Basser1994} and
liquid crystal modelling \cite{DeGennes1993}.  In these applications,
the tensor's eigenvalues typically have a direct physical meaning and
are required to satisfy pointwise bounds. In diffusion tensor imaging,
positive definiteness is essential to represent meaningful diffusion
processes \cite{Arsigny2006}. In computational solid mechanics,
constraints on the eigenvalues of stress and strain are tied to
stability and admissibility of material responses \cite{Simo1992}. In
viscoelastic and other non-Newtonian models the conformation or
extra-stress tensor is positive definite
\cite{ashby2025discretisation} and, in some cases such as FENE-P,
additional constraints such as a bounded trace are imposed, see
\cite{Owens2002}.

Standard discretisations for convection-diffusion-type PDEs do not
typically enforce eigenvalue constraints. In particular, continuous
Galerkin finite element methods \cite{EG21-I,Brenner2008}, finite
volume schemes \cite{LeVeque2002} and spectral methods
\cite{Canuto2007} may generate approximations whose eigenvalues
violate prescribed bounds, especially in convection-dominated regimes
or near sharp fronts. Such violations can lead to non-physical states
and may trigger numerical instabilities or loss of robustness, even
when the underlying continuous model preserves admissibility.

\subsection*{Contributions of this work.}
We develop a finite element method for time-dependent
convection--diffusion equations with symmetric tensor unknowns that
preserves a prescribed eigenvalue range $[\epsilon,\kappa]$, with
$\epsilon < \kappa$, at the degrees of freedom. The method extends
recent nodally bound-preserving finite element ideas from scalar to
tensor-valued problems by combining a baseline stabilised
discretisation with a convex admissible set of tensor-valued finite
element functions whose nodal values satisfy the eigenvalue bounds.
The fully discrete scheme is formulated and analysed as a variational
inequality over this admissible set. For its numerical realisation we
employ an iterative solver based on a nodal projection defined by local
spectral decomposition followed by eigenvalue truncation. For time
discretisation we use implicit Euler, and for space discretisation we
use continuous interior penalty (CIP), which provides stabilisation in
the convection-dominated regime \cite{MR2068903,BF09}. We prove
unconditional stability and derive a priori error estimates for the
fully discrete method. Numerical experiments illustrate that the
approach prevents eigenvalue over- and undershoots observed in the
baseline scheme, while retaining good accuracy for both smooth and
non-smooth data.

\subsection*{Relation to the literature.}
The design of bound-preserving discretisations, often discussed in the
context of the discrete maximum principle (DMP), has a long history.
Early analytical results for maximum principles in finite element
methods date back to \cite{CR73}. In general, however, standard
Galerkin schemes on arbitrary meshes do not satisfy the DMP for
convection--diffusion problems, leading to spurious oscillations and
unphysical overshoots. This difficulty motivated a range of remedies.
One classical approach is to add artificial diffusion or to employ
upwinding to stabilise the scheme, including early
artificial-viscosity ideas in the finite element setting
\cite{kikuchi1977discrete} and nonlinear Petrov--Galerkin upwind
constructions designed to enforce a maximum principle in
convection-dominated regimes \cite{MH85}. For piecewise linear
elements,  the DMP can be guaranteed under restrictive geometric or
algebraic conditions, for example acute simplicial meshes or an
$M$-matrix structure of the stiffness operator
\cite{brandts2008discrete}. As a consequence, many modern
bound-preserving methods are necessarily nonlinear and rely on added
stabilisation, limiters or algebraic corrections to eliminate new
extrema, consistent with the classical limitations on linear
high-order monotone schemes. Representative finite element approaches
include CIP-type stabilisation augmented by nonlinear diffusion to
recover a DMP \cite{BE05} and edge-based nonlinear diffusion
mechanisms closely connected to algebraic flux-correction techniques
\cite{BBK17-NumMath}. In parallel, the flux-corrected transport (FCT)
and algebraic flux-correction (AFC) paradigms modify the algebraic
form of the discrete operator by limiting antidiffusive fluxes so as
to enforce bounds while retaining accuracy \cite{Kuz07}. Variants and
extensions of these ideas, including multiscale and
constraint-enforcement perspectives, are discussed in \cite{EHS09}. A
broad and up-to-date account of monotonicity-preserving and
DMP-related finite element methodology can be found in the recent
monograph \cite{BJK25} and the references therein.

Beyond maximum-principle preservation for scalar problems, a closely
related body of work concerns invariant-domain and
structure-preserving discretisations for systems, where admissibility
is expressed through convex invariant sets, entropy inequalities or
energy dissipation.  In the discontinuous Galerkin setting, such
questions are often addressed through a combination of stabilisation
and limiting, with analysis frequently organised around entropy or
relative-entropy frameworks, see for instance
\cite{giesselmann2015posteriori,giesselmann2016reduced}.  Although
these developments do not directly target tensor eigenvalue
constraints, they provide a useful methodological backdrop, they
highlight how nonlinear admissibility requirements typically enter
either through limiters applied to a baseline high-order method or
through constrained variational principles. 

For {\em tensor-valued} problems, the literature on bound-preserving
schemes is far more limited. To the best of our knowledge, the main
existing contributions that explicitly preserve eigenvalue ranges for
transported symmetric tensors are \cite{lohmann2017flux,lohmann2019algebraic}.
These works extend FCT and AFC ideas to matrix-valued unknowns, with
limiting strategies designed to prevent the loss of key admissibility
properties such as positive definiteness during transport. Beyond
these contributions, eigenvalue control for tensor quantities is more
commonly addressed indirectly, for example through problem-specific
reconstructions or limiters in hyperbolic and remap-type algorithms,
rather than through a systematic finite element framework.

A different line of work enforces admissibility through nonlinear
transformations rather than direct limiting. In the context of
symmetric positive-definite (SPD) tensor fields, geometric
constructions based on the matrix logarithm provide parametrisations
in which SPD is preserved by construction \cite{Arsigny2006,Moakher2005}.
In non-Newtonian fluid mechanics, the log-conformation representation
is a prominent example of this philosophy \cite{FK04}. While such
transformations can be effective, they introduce additional nonlinear
structure and can complicate both analysis and implementation when
combined with standard stabilisation and time-stepping strategies,
particularly in convection-dominated regimes.

The present work is most closely related to recent developments on
{\em nodally} bound-preserving finite element methods for scalar
problems. The framework in \cite{BGPV23} constrains nodal degrees of
freedom to lie in an admissible range by projection onto a convex
nodal set and, in certain settings, admits an interpretation as a
variational inequality. This viewpoint aligns with the broader theory
of constrained energy minimisation and variational inequalities that
underpins many finite element treatments of inequality constraints
\cite{Stampacchia,kirby2024high,keith2024proximal}. Time-dependent
extensions to reaction--convection--diffusion and related parabolic
problems have been developed in
\cite{amiri2024nodally,amiri2025nodally}, with the essential mechanism
remaining a baseline discretisation coupled to a nonlinear
admissibility enforcement at each time level.

These nodal approaches sit alongside a closely related line of work on
invariant-domain and convex-limiting strategies for continuous finite
element discretisations of hyperbolic systems, where admissibility is
formulated in terms of convex invariant sets and recovered a
posteriori by limiting a high-order update towards a low-order
invariant-domain update
\cite{guermond2014second,guermond2016invariant,guermond2019invariant}.
In the discontinuous Galerkin setting, a complementary and extensive
literature establishes positivity-preserving and invariant-domain
ideas via cell-average constraints and limiters compatible with SSP
time integrators \cite{zhang2010positivity,zhang2011positivity}.
Although these strands differ in their admissibility variables (nodal
values versus cell averages or local invariant sets) and in the
underlying discretisation class (continuous versus discontinuous
Galerkin), they share the same organising principle, a baseline scheme
is augmented by a nonlinear correction that enforces an admissible set
without destroying accuracy away from constrained regions.

Further developments of nodal admissibility ideas beyond the original
scalar convection--diffusion setting include extensions to
drift--diffusion-type models and to hyperbolic convection--reaction
problems, illustrating the flexibility of the projection/constraint
paradigm across different PDE structures
\cite{barrenechea2025nodally,ashby2025nodally}.  Here, this line of
ideas is extended from scalar to tensor-valued convection--diffusion
equations by formulating a variational inequality directly on an
admissible set defined through nodal eigenvalue constraints, using a
CIP baseline discretisation. This provides a route to enforce
physically motivated tensor constraints while retaining a standard
finite element setting and allowing convection-dominated
stabilisation.

The paper is organised as follows. In Section \ref{sec:pre} we
introduce notation and preliminary results. In Section~\ref{The
  modelProblem} we present the model problem, the finite element
spaces, the key inequalities used in the analysis and the admissible
convex set. Section \ref{Sec:FEM} introduces the fully discrete
variational inequality formulation. Stability and convergence are
proved in Section \ref{Sec:Error}. Numerical experiments are presented
in Section \ref{sec:numerics}.  Concluding remarks and possible
extensions are given in Section \ref{sec:conc}.

\section{Notation and Preliminaries}
\label{sec:pre}

Let us consider the space of symmetric tensors of dimension $d \times d$, $1 \leq d \leq 3$.
We denote this space by $\mathbb{S}_d \subset \mathbb{R}^{d \times d}$, while $\mathbb{S}_d^+$ stands for the subset of positive semidefinite tensors.
Table~\ref{tab:notation} summarises the symbols frequently used in this work.

\begin{table}[h!]
	\centering
	\caption{Summary of symbols frequently used in this work.}
	\label{tab:notation}
	\begin{tabular}{|c|l|}
			\hline
			\textbf{Symbol} & \textbf{Description} \\ \hline
			$d$             & Dimension of the tensor \\ 
			$m$             & Dimension of the space (i.e., $\Omega\subset\mathbb{R}^m$) \\ 
			$k,\ell$        & Indices corresponding to tensor components, i.e. $1 \leq k,\ell \leq d$ \\ 
			$\mathbb{S}_d$  & Space of symmetric tensors in $\mathbb{R}^{d \times d}$ \\ 
			$\mathbf{0}$    & Zero tensor \\ 
			$\mathbf{I}$    & Identity tensor  \\ 
			$v_{k\ell}$     & Tensor entry of $\mathbf{V}\in \mathbb{S}_d $ \\
			$\tilde{\mathbf{V}}=\mathrm{diag}(\lambda_{1}, \lambda_{2},\ldots,\lambda_{d})$ & Diagonal tensor with eigenvalues of $\mathbf{V}\in \mathbb{S}_d$ on its diagonal \\ 
			$\mathbf{x}$    & Vector in $\mathbb{R}^{m}$ \\ 
			$\mathbf{V} =\mathbf{Q}_{\bV}\tilde{\mathbf{V}}\mathbf{Q}^{T}_{\bV}$ & Spectral decomposition of $\mathbf{V} \in \mathbb{S}_d$  
		\\ \hline
		\end{tabular}
\end{table}

Additionally, we adopt the notation that lowercase bold letters such as $\mathbf{v}\in \mathbb{R}^{d}$ denote $d$-dimensional vectors, while uppercase bold letters, such as $\mathbf{U}$ and $\mathbf{V}$, are reserved for tensor quantities.

Since all the tensors in this manuscript are symmetric, they are diagonalisable. So, they admit the
decomposition
\begin{equation}
	\mathbf{V}=\mathbf{Q}_{\bV}\tilde{\mathbf{V}}\mathbf{Q}^{T}_{\bV}\quad\textrm{where}\quad
	\tilde{\mathbf{V}}:={\rm diag}(\lambda_{1}, \lambda_{2},\ldots,\lambda_{d})\quad,\quad   \mathbf{Q}_{\bV}^{T}\mathbf{Q}_{\bV}=\mathbf{I}.
\end{equation}
Above, as standard, $\tilde{\mathbf{V}}$ denotes the diagonal tensor of sorted eigenvalues $\lambda_{k}(\mathbf{V})$ of $\mathbf{V}$, and $\mathbf{Q}_{\bV}$ is the tensor of corresponding eigenvectors of $\mathbf{V}$.
Whenever it is clear from the context, we omit $\mathbf{V}$
from $\lambda_{k}(\mathbf{V})$ for simplicity and we use $\lambda_{k}$.

Since this work will focus on preserving the eigenvalue range of symmetric tensors, for $\epsilon\ge 0$ 
 and  $\kappa> \epsilon$, we  define the following convex subset of $\mathbb{S}_d$ 
\begin{align}
	\mathbb{S}_{d}^{\epsilon,\kappa}:=\{\mathbf{V}\in \mathbb{S}_d :\epsilon	\leq \lambda_{1}^{}\leq \lambda_{2}^{}\leq\ldots\leq \lambda_{d}^{}\leq \kappa \}.\label{eq161}
\end{align} 
(This is a closed convex set in $\mathbb{S}_d$.)
Here  $\lambda_{k}^{}$, $k=1,\ldots,d$ are sorted eigenvalues of $\mathbf{V}$. 
Using this convex subset, we decompose 
every element $ \mathbf{V}\in \mathbb{S}_d$  as the sum
$\mathbf{V}=\mathbf{V}^{+}+\mathbf{V}^{-}$, where $\mathbf{V}^{+}$ and $\mathbf{V}^{-}$ are given by
\begin{align}
	\mathbf{V}^{+}= \mathbf{Q}_{\bV}\tilde{\mathbf{V}}^{+}\mathbf{Q}^{T}_{\bV},\label{Vplus}
\end{align}
where the $k$th diagonal entry of $\tilde{\mathbf{V}}^{+}$ is defined as 
\begin{align}
(\lambda_{k}(\mathbf{V}))^{+}=\max\Big\{\epsilon,\min\{\lambda_{k}(\mathbf{V}),\kappa\}\Big\} \label{lamoperator}
\end{align}
 and 
\begin{align}
	\mathbf{V}^{-}=\mathbf{V}-\mathbf{V}^{+}. \label{Vminus}
\end{align}

In this paper we will make use of the Frobenius norm \( \|\cdot\|_F \). Therefore, we summarise some of its most important properties. According to the invariance of the trace \(\text{tr}(\cdot)\) under cyclic permutations,
\[
\text{tr}(\mathbf{V}\mathbf{W}) = \text{tr}(\mathbf{W}\mathbf{V}), \quad \text{for all } \mathbf{V}, \mathbf{W} \in \mathbb{R}^{d \times d},
\]
and the definition of the Frobenius inner product \((\cdot, \cdot)_F\),
\begin{equation}
(\mathbf{V}, \mathbf{W})_F := \mathbf{V} : \mathbf{W} = \sum_{k,\ell=1}^d v_{k\ell}w_{k\ell} = \text{tr}(\mathbf{V}^\top \mathbf{W}) = \text{tr}(\mathbf{W}^\top \mathbf{V}), \label{Ferobino}
\end{equation}
the Frobenius norm \( \|\cdot\|_F \)  satisfies the identity
\begin{equation}
\|\mathbf{V}\|_F^2 := (\mathbf{V}, \mathbf{V})_F = \text{tr}(\mathbf{V}^\top \mathbf{V}) = \text{tr}((\mathbf{Q}^\top \mathbf{V}\mathbf{Q})^\top \mathbf{Q}^\top \mathbf{V}\mathbf{Q}) = \sum_{k=1}^d \lambda_k^2, \quad \text{for all } \mathbf{V} \in \mathbb{S}_d. \label{2.1}
\end{equation}

We will adopt standard notations for Sobolev spaces  in line with,
e.g., \cite{EG21-I}. For $D\subseteq\mathbb{R}^{m}$, $m=1,2,3$, we denote by
$\|\cdot\|_{0,p,D} $ the $L^{p}(D)$-norm; when $p=2$ the subscript $p$
will be omitted and we only write $\|\cdot\|_{0,D} $. In addition, for
$s\geq 0$, $p\in [1,\infty]$, we denote by $\| \cdot \|_{s,p,D}$ ($|
\cdot |_{s,p,D}$) the norm (seminorm) in $W^{s,p}(D)$; when $p=2$, we
define $H^{s}(D)=W^{s,2}(D)$, and again omit the subscript $p$ and only write $\|\cdot \|_{s,D}$
($| \cdot |_{s,D}$).  The following space will also be used repeatedly within the text
\begin{equation}
	H^{1}_{0}(D)=\left\{ v \in H^1(D) : v = 0 \;  {\rm on} \;  \partial D \right\}, \label{space}
\end{equation}
and the space $H^{-1}(D)$ which is the dual of $H^{1}_{0}(D)$. 

For $1\leq p\leq +\infty$, $L^{p}((0,T);W^{s,p}(D))$ is the space defined by
\begin{equation*}
	L^{p}((0,T);W^{s,p}(D))=\left\{u(t,\cdot)\in W^{s,p}(D)~~\text{for almost all $t \in [0,T]$}  : t\mapsto \| u(t,\cdot) \|_{s,p,D}\in L^{p}(0,T)\right\} ,
\end{equation*}
which is a Banach space for the norm
\begin{equation*}
	\begin{split}	
		\|	u \|_{L^{p}((0,T);W^{s,p}(D))}	=	\left\{ \begin{array}{ll} \left(\int_{0}^{T}\| u \|_{s,p,D}^{p} dt\right)^{\frac{1}{p}} \hspace{2.3cm}{\rm if}  \hspace{0.2cm}1\leq p < \infty,\\
			\\
			{\rm ess} \hspace{0.05cm}\sup_{t\in (0,T)}\| u \|_{s,p,D}\hspace{1.9cm}{\rm if}\hspace{0.3cm} p=\infty.
		\end{array} \right.
	\end{split}		
\end{equation*}

The extension of the Sobolev norms to the vector and tensor-valued cases is straightforward. In fact,
the inner product in $L^2(D)^{d\times d}$ is defined as
\begin{equation}
	(\mathbf{U}, \mathbf{V})_{D} = \sum_{k,\ell=1}^d \int_D u_{k\ell}v_{k\ell}   {\rm d}\bx =\int_D \mathbf{U}: \mathbf{V} {\rm d}\bx,\label{innertensor}
\end{equation}
which induces the norm 
$ \|\mathbf{U}\|_{0,D} =(\mathbf{U}, \mathbf{U})_{D}^{\frac{1}{2}}$.   Using similar definitions for the derivatives
we can extend the Sobolev norms to vector and tensor-valued quantities. Finally, we will use the following
Sobolev space
\begin{equation}
	(H^{1}_{0}(D))^{d\times d,\textrm{sym}}=\left\{ \mathbf{V} \in (H^1_0(D))^{d\times d} : \mathbf{V}\in \mathbb{S}_d^{}\; \textrm{a.e. in}  \; D\right\} .\label{space-sym}
\end{equation}

\section{The model problem} \label{The modelProblem}

Let $\Omega$ be an open bounded Lipschitz domain in $\mathbb{R}^{m}$ ($m=2,3$) with polyhedral boundary $\partial \Omega$, and $T>0$.
For a given $\mathbf{F}\in (L^{2}((0,T);L^{2}(\Omega)))^{d\times d}$, we consider the following convection-diffusion problem:
\begin{equation}
	\begin{split}
		\left\{
		\begin{aligned}
			\partial_{t}\mathbf{U} - {\rm div}\big(\mathcal{D}  \nabla \mathbf{U}\big) + \bbeta \cdot \nabla \mathbf{U} + \mu \mathbf{U} &= \mathbf{F}  &&\text{in } (0,T] \times \Omega, \\
			\mathbf{U}(\bx, t) &= 0  &&\text{on } (0,T] \times \partial \Omega, \\
			\mathbf{U}(\cdot, 0) &= \mathbf{U}^{0} &&\text{in } \Omega,
		\end{aligned}
		\right.\label{CDR}
	\end{split}
\end{equation}
where $\mathcal{D}=(d_{ij})_{i,j=1}^{m}\in [L^{\infty}(\Omega)]^{m\times m}$ is symmetric and uniformly strictly positive definite a.e.~in $\Omega$, $\bbeta=( \beta_{i})_{i=1}^{m}\in L^{\infty}((0,T);W^{1,\infty}(\Omega))^{m}$,  and $\mu\in \mathbb{R}^{+}_{0}$, respectively, are the diffusion coefficient, the convective field, and the reaction coefficient. 
We will assume that $ {\rm div} \bbeta=0$ in $\Omega\times[0,T]$.

\begin{Remark}[Componentwise interpretation of the transport terms]
	In equation \eqref{CDR}, the diffusion term ${\rm div}(\mathcal{D}  \nabla \mathbf{U})$ involves the action of the tensor $\mathcal{D} \in \mathbb{R}^{m \times m}$ on $\nabla \mathbf{U} \in \mathbb{R}^{d \times d\times m}$  which is a multilinear operator (see e.g. \cite{amiri2021approximation}). The product $\mathcal{D}  \nabla \mathbf{U}$ is understood as the following tensor product 
	\begin{equation}
		\mathcal{D} \nabla \mathbf{U}   =  \begin{pmatrix}
		\mathcal{D}\nabla U_{11} & \cdots & \mathcal{D}\nabla U_{1d} \\
		\vdots & \ddots & \vdots \\
		\mathcal{D}\nabla U_{d1} & \cdots & \mathcal{D}\nabla U_{dd}
	\end{pmatrix},
	\end{equation}
	so that the $(k,\ell)$-th component of the diffusion term is ${\rm div}\big( \mathcal{D}\nabla U_{k\ell}\big)$. The convection term $\bbeta \cdot \nabla \mathbf{U} $ is defined similarly, i.e. $(\bbeta\cdot\nabla \mathbf{U})_{k\ell}=\bbeta\cdot\nabla U_{k\ell}$.
\end{Remark}

The standard weak formulation of  \eqref{CDR} reads as follows: Find
$\mathbf{U}\in L^\infty((0,T);(H^1_0(\Omega))^{d\times d,\textrm{sym}})\cap\left( H^1((0,T);H^{-1}(\Omega))\right)^{d\times d,\textrm{sym}}$ such that, 
for almost all $t\in (0,T)$ the following holds
\begin{equation}
	\begin{split}
		\left\{ \begin{array}{ll}  
			(\partial_{t}\mathbf{U},\mathbf{V})_{\Omega} + a(\mathbf{U},\mathbf{V}) = (\mathbf{F},\mathbf{V})_{\Omega} \hspace{1cm} \forall \mathbf{V}\in 	\left(H^{1}_{0}(\Omega)\right)^{d\times d, \textrm{sym}},\\
			\hspace{1.6cm} \mathbf{U}(\cdot,0) = \mathbf{U}^{0}.
		\end{array} \right.\label{eq82}
	\end{split}
\end{equation}
Here, the bilinear form $a(\cdot,\cdot)$ is defined by
\begin{equation}
	a(\mathbf{W},\mathbf{V}):=\left(\mathcal{D}\nabla \mathbf{W},\nabla \mathbf{V}\right)_{\Omega}+(\bbeta\cdot\nabla \mathbf{W},\mathbf{V})_{\Omega}+\mu( \mathbf{W},\mathbf{V})_{\Omega}.\label{eq81}
\end{equation}
In the above definition we have slightly abused the notation, as the convective term $\bbeta$ might depend on $t$, but unless the context requires it, we will always denote
this bilinear form by $a(\cdot,\cdot)$.  Since we have supposed that $\bbeta$ is solenoidal, then the bilinear form $a(\cdot,\cdot)$ is elliptic, in the sense that for $\mathbf{V}\in (H^1_0(\Omega))^{d\times d}$ it holds that $(\bbeta\cdot\nabla\mathbf{V},\mathbf{V})_\Omega=0$ and hence $a(\mathbf{V},\mathbf{V})=(\mathcal{D}\nabla\mathbf{V},\nabla\mathbf{V})_\Omega+\mu(\mathbf{V},\mathbf{V})_\Omega$.
More precisely, for each $t\in (0,T)$ the 
bilinear form  $a(\cdot,\cdot)$ induces the following ``energy'' norm in $	\left(H^{1}_{0}(\Omega)\right)^{d\times d}$
\begin{equation*}
	\|\mathbf{V} \|_{a}=\left( a(\mathbf{V},\mathbf{V})\right)^{\frac{1}{2}},\qquad t\in [0,T].
\end{equation*}

It is a well-known fact that the problem \eqref{eq82} has a unique solution $\mathbf{U}$, as a consequence
of Lions' Theorem (see, e.g., \cite[Theorem~4.1]{lions2012non}).  Building on what was mentioned in the introduction, we will make the following assumption on  $\mathbf{U}$.

\noindent\underline{Assumption (A1):} We will suppose that the eigenvalues of the weak solution of \eqref{eq82} satisfy 
\begin{equation}
	\epsilon\leq \lambda_{k}(\bU(\bx,t))\leq \kappa, \hspace{0.5cm}k=1,\cdots,d,\hspace{0.5cm}\text{ for almost all}\ \  (\bx,t)\in \Omega\times(0,T),\label{eq14}
\end{equation}
where $\epsilon$ and $\kappa$ are known non-negative constants.

\subsection{Space discretisation}
Since in problem \eqref{eq82} the space and time variables play different roles, we first approximate the solution in \eqref{eq82} only in space, reducing it to a system of coupled ordinary differential equations where time is the only independent variable.

To discretise \eqref{eq82} with respect to space in tensor form, we consider a finite-dimensional subspace of \((H^{1}(\Omega))^{d\times d}\). Let \(\calP\) be a conforming and shape-regular partition of \(\Omega\) into simplices (or affine quadrilateral/hexahedra).  We denote by $\boldsymbol{x}_1^{},\ldots,\boldsymbol{x}_M^{}$ the interior nodes
of  \(\calP\).
Over \(\calP\), and for \(k \geq 1\), we define the tensor finite element space as
\begin{align}
	\mathbb{V}_\calP := \{\mathbf{V}_h \in (C^{0}(\overline{\Omega}))^{d\times d} : \mathbf{V}_h|_{K} \in (\mathfrak{R}(K))^{d\times d}, \ \forall K \in \calP \} \cap (H_{0}^{1}(\Omega))^{d\times d, \textrm{sym}}, \label{eq1_tensor}
\end{align}
where
\begin{align}
	\mathfrak{R}(K) = 
	\begin{cases} 
		\mathbb{P}_{k}(K), & \text{if } K \text{ is a simplex}, \\
		\mathbb{Q}_{k}(K), & \text{if } K \text{ is an affine quadrilateral/hexahedral},\label{simplexquadri}
	\end{cases}
\end{align}
with \(\mathbb{P}_{k}(K)\) denoting the polynomials of total degree \(k\) on \(K\) and \(\mathbb{Q}_{k}(K)\) denoting the mapped space of polynomials of degree at most \(k\) in each variable.

\begin{Remark}[Tensor-valued basis construction]
There are several ways in which the basis functions for this space can be built.  We now give some
more details on this process. 
Let $V_\calP$ denote the scalar finite element space defined as
\begin{align}
	V_\calP := \{v_h \in C^0(\overline{\Omega}) : v_h|_K \in \mathfrak{R}(K), \forall K \in \calP\} \cap H^1_0(\Omega),
\end{align}
with basis functions $\{\phi_i\}_{i=1}^M$.
Therefore, based on the definition of the finite element space $\mathbb{V}_{\calP}$, any function $\bV_{h} \in \mathbb{V}_{\calP} $ can be represented by the following expansion
\begin{align}
	\bV_{h}=\sum_{i=1}^{M} \sum_{1 \leq j \leq \ell \leq d}v_{j\ell}^i\Phi^{j\ell}_{i}, \label{lagrange}
\end{align}
where $\Phi^{j\ell}_{i}$ are the tensor-valued basis functions given by
\begin{align}
	\Phi^{j\ell}_{i}=
	\begin{cases}
		\phi_{i}\mathbf{e}_j \otimes \mathbf{e}_j, & \text{if } j = \ell, \\
		\phi_{i}(\mathbf{e}_j \otimes \mathbf{e}_{\ell} + \mathbf{e}_{\ell} \otimes \mathbf{e}_j), & \text{if } j < \ell.
	\end{cases} \label{eq:basisfun}
\end{align}
Here,   $\mathbf{e}_j$ is the $j$-th canonical basis vector in $\mathbb{R}^d$, $v_{j\ell}^i$ are the coefficients of the expansion, and $\otimes$ denotes the tensor product. The tensor-valued basis functions $\Phi^{j\ell}_{i}$ are constructed to preserve the symmetry of the elements of $\mathbb{V}_{\calP}$.
In particular, for $j<\ell$ the symmetric combination in \eqref{eq:basisfun} ensures that $\Phi^{j\ell}_{i}$ has equal $(j,\ell)$ and $(\ell,j)$ components.
Note that since the tensor space $\mathbb{V}_{\calP}$ consists of symmetric tensors, its dimension is $M \frac{d(d+1)}{2}$.
\end{Remark}

The natural extension of the Lagrange interpolation operator (see \cite[Chapter~11]{EG21-I} for its definition) to the tensor-valued context is the following mapping
 	\begin{align}
 		\mathbf{I}_{h}:\big(\mathcal{C}^{0}(\overline{\Omega})\big)^{d\times d}\cap \big( H^1_0(\Omega)\big)^{d\times d,\textrm{sym}}&\longrightarrow \mathbb{V}_{\mathcal{P}},\nonumber\\
 		\mathbf{V}&\longmapsto \mathbf{I}_{h}\mathbf{V}:= \sum_{i=1}^{M} \sum_{k,\ell=1}^{d} v_{k\ell}(\bx_i)\Phi^{k\ell}_{i} .\label{eq:componentwise_interpolation}
 	\end{align}
(Here the interpolation is defined componentwise by evaluating $\mathbf{V}$ at the nodal points $\bx_i$ and expanding in the tensor-valued basis from \eqref{eq:basisfun}.)
It satisfies the following approximation property (see \cite[Proposition~1.12]{EG21-I}): Let $1\leq \ell \leq k$. Then there exists $C>0$, 
independent of $h$, such that for all 
$\mathbf{V} \in (H^{\ell+1}(\Omega))^{d\times d}\cap (H^1_0(\Omega))^{d\times d}$,
\begin{equation}
	\|\mathbf{V}- \mathbf{I}_{h}\mathbf{V}\|_{0,K}
	+ h_{K} |\mathbf{V}-\mathbf{I}_{h}\mathbf{V}|_{1,K}
	\leq C h^{\ell +1}_{K}|\mathbf{V}|_{\ell +1,K}.
	\label{tensorlagrange}
\end{equation}

In addition, we recall some standard estimates for finite element functions, presented
here in the tensor-valued form. First, we recall the following inverse inequality
(see \cite[Lemma~12.1]{EG21-I}): for all $s,\ell\in \mathbb{N}_0$, $0\le s\leq \ell$ and all $p,q \in [1,\infty]$, 
there exists a constant $C>0$, independent of $h$, such that for all 
$\mathbf{V}_h \in \mathbb{V}_\calP$
\begin{equation}
	|\mathbf{V}_h|_{\ell,p,K} \leq 
	C h_K^{s-\ell+m\left(\frac{1}{p}-\frac{1}{q}\right)} |\mathbf{V}_h|_{s,q,K}.
	\label{inversetensor}
\end{equation}
(The exponent involves the spatial dimension $m$ of $\Omega\subset\mathbb{R}^m$, rather than the tensor dimension $d$.)
In addition, we recall the following discrete trace inequality (see \cite[Lemma~2.15]{EG21-I}): there exists $C>0$ independent of $h$ such that, for every 
$\mathbf{V} \in (H^1(K))^{d\times d}$,
\begin{equation}
	\|\mathbf{V}\|^{2}_{0,\partial K} 
	\leq C \left(h_K^{-1}\|\mathbf{V}\|^2_{0,K}+h_K |\mathbf{V}|^{2}_{1,K}\right).
	\label{Tensortrace}
\end{equation}

\subsection{The baseline discretisation}

As discussed in the introduction, the method is built over a baseline discretisation. For convection-dominated (or transport) problems, it is a well-established fact that the plain Galerkin method should not 
be used. So, in this work the baseline discretisation is a stabilised finite element method. In principle,
any stabilised method can be used, but to fix ideas in this work our method of choice
is Continuous Interior Penalty (CIP), originally proposed in \cite{MR2068903} and analysed in detail for the time-dependent problem in \cite{BF09}. The CIP 
method adds the following stabilising term to the Galerkin scheme:
\begin{equation}
	J(\bU_h^{},\bV_h^{})=\gamma \sum_{F\in\calF_I^{}}\int_F\|\bbeta \|_{0,\infty,F}^{} h_{F}^{2}\llbracket  \boldsymbol \nabla \bU_{h} \rrbracket :\llbracket \boldsymbol \nabla \bV_{h}  \rrbracket   \mathrm{d}s . \label{eq10}
\end{equation}
(Here $\calF_I$ denotes the set of interior faces of the mesh, $h_F$ is a characteristic diameter of $F$, and $\llbracket \nabla \mathbf{V}_h\rrbracket$ denotes the jump of the broken gradient across $F$ in the standard CIP sense.)
Thus, the stabilised method reads as follows: 
\begin{equation}
	\begin{split}
		\left\{ \begin{array}{ll}
			\text{For almost all $t\in  (0,T)$, find $\mathbf{U}_{h}\in \mathbb{V}_{\calP}$ such that} \vspace{.1cm}
			\\  
			(\partial_{t}\mathbf{U}_{h},\mathbf{V}_{h})_{\Omega} + a_{J}(\mathbf{U}_{h},\mathbf{V}_{h}) = (\mathbf{F},\mathbf{V}_{h})_{\Omega} \hspace{1cm} \forall \mathbf{V}_{h}\in \mathbb{V}_\calP,  \vspace{.1cm}\\
			\hspace{2.6cm} \mathbf{U}_{h}(\cdot,0) = \mathbf{I}_{h}\mathbf{U}^{0},
		\end{array} \right.\label{eq824}
	\end{split}
\end{equation}
where
\begin{equation}
	a_{J}(\mathbf{U}_{h},\mathbf{V}_{h}):=a(\mathbf{U}_{h},\mathbf{V}_{h})+	J(\mathbf{U}_{h},\mathbf{V}_{h}) .\label{eq12}
\end{equation}

The analysis of this method relies on the fact that the bilinear form $a_{J}(\cdot,\cdot)$ is elliptic. In fact,
it satisfies the following: for all $\bV_h\in\mathbb{V}_\calP^{}$ we have
\begin{equation}\label{aJ-ellipt}
a_{J}(\mathbf{V}_{h},\mathbf{V}_{h})=\|\bV_h^{}\|_{a}^2+J(\bV_h^{},\bV_h^{})=: \|\bV_h^{}\|_{a_J}^2 .
\end{equation}

\subsection{The admissible set}

We introduce the following \textit{admissible set}, that is, the set of finite element
functions such that, at each degree of freedom,  they belong to  $\mathbb{S}_{d}^{\epsilon,\kappa}$. That
is,
\begin{align}
	\mathbb{V}_{\calP}^{\epsilon,\kappa}:=\{\mathbf{V}_{h}\in \mathbb{V}_{\mathcal{P}}: \mathbf{V}_{h}(\boldsymbol{x}_{i})\in \mathbb{S}_{d}^{\epsilon,\kappa},  \text{for all}\ \ i=1,\ldots,M  \}.\label{eq16}
\end{align} 

\begin{Remark}[Convexity of the admissible set]
One very important property of this set is that it is convex. In fact, given two elements
$\mathbf{U}_{h},\mathbf{V}_{h}\in\mathbb{V}_{\mathcal{P}}^{\epsilon,\kappa}$ and
$t\in [0,1]$ we see that, for every $\mathbf{x}\in\mathbb{R}^d$ we have
\begin{equation*}
\mathbf{x}^T\big( t\mathbf{U}_{h}+(1-t)\mathbf{V}_{h}\big)\mathbf{x}
= t \mathbf{x}^T\mathbf{U}_{h}\mathbf{x} + (1-t)  \mathbf{x}^T\mathbf{V}_{h}\mathbf{x}
\ge t\epsilon \mathbf{x}^T\mathbf{x}+(1-t)\epsilon \mathbf{x}^T\mathbf{x}
= \epsilon \mathbf{x}^T\mathbf{x} ,
\end{equation*}
and thus the minimum eigenvalue of $t\mathbf{U}_{h}+(1-t)\mathbf{V}_{h}$ is larger than,  or equal to,
$\epsilon$.  In a similar fashion we can prove that the largest eigenvalue of $t\mathbf{U}_{h}+(1-t)\mathbf{V}_{h}$ is smaller than, or equal to $\kappa$.\hfill$\Box$
\end{Remark}

We finish this section by presenting an algebraic projection onto $\mathbb{V}_\calP^{\epsilon,\kappa}$
that will be used in the iterative scheme employed in our numerical experiments.
Using the definitions \eqref{Vplus} and \eqref{Vminus} at each nodal value, we split $\mathbf{V}_{h}(\bx_{i})$ as 
$\mathbf{V}_{h}(\bx_{i})=\mathbf{V}_h(\bx_{i})^{+}+\mathbf{V}_h(\bx_{i})^{-}$, and thus we can define $\mathbf{V}_{h}^{+}$ and $\mathbf{V}_{h}^{-}$ as
\begin{align}
\mathbf{V}_{h}^{+}=\sum_{i=1}^{M} \sum_{1 \leq k \leq \ell \leq d} v_{k\ell}^{+}(\bx_i) 	\Phi^{k\ell}_{i} ,\quad \label{eq17}
\end{align}
where $v_{k\ell}^{+}(\bx_i)$ denotes the $(k,\ell)$-entry of the constrained nodal tensor $\mathbf{V}_h(\bx_i)^{+}$, and $\Phi^{k\ell}_{i}$ is the basis function defined in \eqref{eq:basisfun}, and 
\begin{align}
	\mathbf{V}_h^{-}=\mathbf{V}_h-\mathbf{V}_{h}^{+}. \label{eq18}
\end{align}
We refer to $\mathbf{V}_{h}^{+}$ and $\mathbf{V}_{h}^{-}$ as the \textit{constrained}
and \textit{complementary} parts of $\mathbf{V}_{h}$, respectively. Using this
decomposition we define the following algebraic projection
\begin{equation}
	(\cdot)^{+}:\mathbb{V}_{\mathcal{P}}\rightarrow \mathbb{V}_{\calP}^{\epsilon,\kappa}\quad,\quad
	\mathbf{V}_{h}\longmapsto \mathbf{V}_{h}^{+} . \label{posoperator}
\end{equation} 
	
\section{The finite element method}\label{Sec:FEM}

Let $N>0$ be a given positive integer. In what follows, we consider a partition of the time interval $[0,T]$ as $t_{0}=0<t_{1}<t_{2}<\cdots<t_{N}=T$ with the time step size $\Delta t_{n}:=t_{n}-t_{n-1}$. To simplify the notation we assume that the time step size is uniform i.e., $\Delta t_{n}=\Delta t=\frac{T}{N}$. In addition, the discrete value $\bU_{h}^{n}\in \mathbb{V}_\calP$ stands for the approximation of $\bU^{n}=\mathbf{U}(t_{n})$ in $\mathbb{V}_\calP$ for $0\leq n \leq N$.   The discretisation of the time derivative is defined as follows:
\begin{equation*}
 \delta \mathbf{U}_{h}^{n} :=\frac{\bU_{h}^{n}-\mathbf{U}_{h}^{n-1}}{\Delta t} .
 \end{equation*}

With these notations, the finite element method used in this work reads as follows:
\begin{equation}
	\left\{ \begin{array}{ll}
		\text{For $1\leq n \leq N$, find $\mathbf{U}_h^n \in \mathbb{V}_{\calP}^{\epsilon,\kappa}$ such that}\vspace{.1cm}
		\\  
		(\delta \mathbf{U}_{h}^{n},\mathbf{V}_{h}-\mathbf{U}_{h}^{n})_{\Omega} + a_{J}(\mathbf{U}_{h}^{n},\mathbf{V}_{h}-\mathbf{U}_{h}^{n}) \geq (\mathbf{F}^{n},\mathbf{V}_{h}-\mathbf{U}_{h}^{n})_{\Omega} \hspace{1cm} \forall \mathbf{V}_{h}\in \mathbb{V}_{\calP}^{\epsilon,\kappa}, \vspace{.1cm}\\
		\hspace{4.8cm} \mathbf{U}_{h}^{0} = \mathbf{I}_{h}\mathbf{U}^{0} .
	\end{array} \right.\label{eq199}
\end{equation}

The proof of the well-posedness of \eqref{eq199} can be done using Stampacchia's Theorem (\cite[Theorem~2.1]{Stampacchia}).  In fact, it is not difficult to realise that, at each time step $n$, \eqref{eq199} can be rewritten as: Find $\mathbf{U}_h^n \in \mathbb{V}_{\calP}^{\epsilon,\kappa}$, such that
\begin{align}
	\mathcal{B}(\mathbf{U}_h^n, \mathbf{V}_h - \mathbf{U}_h^n) \geq \mathcal{L}(\mathbf{V}_h - \mathbf{U}_h^n) \quad \forall \mathbf{V}_{h}\in \mathbb{V}_{\calP}^{\epsilon,\kappa},\label{variational}
\end{align}
where
\begin{equation}
	\mathcal{B}(\mathbf{W}_h, \mathbf{V}_h) := \frac{1}{\Delta t}(\mathbf{W}_h, \mathbf{V}_h)_{\Omega} +  a_J(\mathbf{W}_h, \mathbf{V}_h),\hspace{0.5cm} \forall \bV_{h}, \bW_{h}\in \mathbb{V}_{\calP}^{}, \label{Bilinear1}
\end{equation}
and 
\begin{equation}
	\mathcal{L}(\mathbf{V}_h) := \frac{1}{\Delta t}(\mathbf{U}_h^{n-1}, \mathbf{V}_h)_{\Omega} + (\mathbf{F}^{n}, \mathbf{V}_h)_{\Omega}.\label{linear1} 
\end{equation}

\begin{theorem}[Well-posedness of the fully discrete variational inequality]
	Let $\mathbf{V}_h \in \mathbb{V}_{\calP}^{\epsilon,\kappa}$.  
	Then the bilinear form $\mathcal{B}(\cdot,\cdot)$ defined in \eqref{Bilinear1} 
	is continuous and elliptic on $\mathbb{V}_{\calP}$.  
	As a consequence, the variational inequality \eqref{variational} admits a unique solution 
	$\mathbf{U}_h^n \in \mathbb{V}_{\calP}^{\epsilon,\kappa}$.
\end{theorem}

\begin{proof}
	Let $\mathbf{V}_h\in\mathbb{V}_\calP^{}$. Then, it follows directly from the definition
	of $\mathcal{B}$ that
	\[
	\mathcal{B}(\mathbf{V}_h,\mathbf{V}_h)=\frac{1}{\Delta t}  \|\mathbf{V}_h\|_{0,\Omega}^2
	+\|\mathbf{V}_h\|^2_{a_J} ,
	\] 
	which shows the ellipticity of $\mathcal{B}$.  In addition,  $\mathbb{V}_{\calP}^{\epsilon,\kappa}$ is a closed convex subset of $\mathbb{V}_{\calP}$. Therefore, Stampacchia’s Theorem 
	yields the existence and uniqueness of solutions to \eqref{variational}.
\end{proof}

\section{Stability and Error analysis}\label{Sec:Error}

This section is devoted to establishing a stability result and deriving optimal error estimates for the method~\eqref{eq199}. A fundamental tool employed throughout the analysis is the discrete Gr\"onwall lemma, originally proved in \cite[Lemma~5.1]{heywood1990finite}.

\begin{lemma}[Discrete Gr\"onwall lemma]\label{13}
	Let $k$, $B$, $a_{n}$, $b_{n}$, $c_{n}$, $\gamma_{n}$, $n=0,\ldots,\nu$, be non-negative numbers such that
	\begin{align*}
		a_{\nu}+k\sum_{n=0}^{\nu}b_{n}\leq k\sum_{n=0}^{\nu}\gamma_{n}a_{n}+k\sum_{n=0}^{\nu}c_{n}+B
		\hspace{0.5cm}{\rm for}\hspace{0.3cm}\nu\geq 0.
	\end{align*}
	Suppose $k\gamma_{n}\leq 1$ for every $n$, and set $\sigma_{n}=(1-k\gamma_{n})^{-1}$. Then
	\begin{align}
		a_{\nu}+k\sum_{n=0}^{\nu}b_{n}\leq \exp\left(k\sum_{n=0}^{\nu}\sigma_{n}\gamma_{n}\right)\left(k\sum_{n=0}^{\nu}c_{n}+B\right).\label{inequality11}
	\end{align}
\end{lemma}

We now prove stability for the fully implicit time discretisation,  
which is precisely the scheme~\eqref{eq199}.  

\begin{lemma}[Energy stability]
	Let $\bU^{n}_{h}\in  \mathbb{V}_\calP$, for $n=1,\ldots, N$ solve \eqref{eq199} for $N\ge 2$ (equivalently, $\Delta t<T$). Then the following stability estimate holds true:
	\begin{equation}
	\begin{aligned}\label{Stabilityestimate}
		\| \mathbf{U}_{h}^{N} \|_{0,\Omega}^{2}
		&+ 2\Big(\sum_{n=1}^{N}\| \mathbf{U}_{h}^{n} - \mathbf{U}_{h}^{n-1} \|_{0,\Omega}^{2}
		+  2 \Delta t \sum_{n=1}^{N}a_{J}(\mathbf{U}_{h}^{n}, \mathbf{U}_{h}^{n}) \Big)
		\\
		&\leq
		\exp\!\Big(\frac{N}{N-1}\Big)\Big(
		4\Delta t\sum_{n=1}^{N}\Big(
		T(\epsilon\mu \sqrt{d}\lvert \Omega \rvert^{\frac{1}{2}}+\|\mathbf{F}^{n}\|_{0,\Omega})^{2}
		+\epsilon \sqrt{d} \lvert \Omega \rvert^{\frac{1}{2}}\|\mathbf{F}^{n}\|_{0,\Omega}\Big)
		\\
		&\hspace{2.35cm}
		+4d \epsilon^{2}\lvert \Omega \rvert
		+2\parallel  \mathbf{U}_{h}^{0}\parallel_{0,\Omega}^{2}
		-4\epsilon\int_\Omega{\rm tr}(\mathbf{U}_{h}^{0}){\rm d}\bx\Big).
	\end{aligned}
	\end{equation}
\end{lemma}

\begin{proof}
	We use the test function $\bV_h^{}=\epsilon \mathbf{I} \in \mathbb{V}_{\calP}^{\epsilon,\kappa}$ in \eqref{eq199} and get
	\begin{equation}\label{Sta-first}
	(\delta (\mathbf{U}_{h}^{n}),\epsilon \mathbf{I}-\mathbf{U}_{h}^{n})_{\Omega}
	+ a_{J}(\mathbf{U}_{h}^{n},\epsilon\mathbf{I}-\mathbf{U}_{h}^{n})
	\geq (\mathbf{F}^{n},\epsilon\mathbf{I}-\mathbf{U}_{h}^{n})_{\Omega}.
	\end{equation}
	For the first term we use the identity $(a-b)a=(a^2-b^2+(a-b)^2)/2$ and obtain
	\begin{equation}\label{Sta-second}
	(\mathbf{U}_{h}^{n}-\mathbf{U}_{h}^{n-1}, \epsilon \mathbf{I}-\mathbf{U}_{h}^{n})_{\Omega}
	= (\mathbf{U}_{h}^{n}-\mathbf{U}_{h}^{n-1}, \epsilon \mathbf{I})_\Omega
	-\frac{1}{2} \left(  \| \mathbf{U}_{h}^{n} \|_{0,\Omega}^{2}-\| \mathbf{U}_{h}^{n-1} \|_{0,\Omega}^{2}
	+ \| \mathbf{U}_{h}^{n} - \mathbf{U}_{h}^{n-1} \|_{0,\Omega}^{2} \right).
	\end {equation}
	To treat the second term in \eqref{Sta-first} we note that $\nabla \mathbf{I}=\mathbf{0}$ and, since ${\rm div} \bbeta=0$ and $\bU_h^n\in (H_0^1(\Omega))^{d\times d,\mathrm{sym}}$,
	the convection contribution satisfies $(\bbeta\cdot\nabla\bU_h^n,\mathbf{I})_\Omega=0$ by integration by parts.
	Therefore
	\begin{equation}\label{Sta-third}
	\epsilon a_{J}(\mathbf{U}_{h}^{n}, \mathbf{I})
	=\epsilon\left(\underbrace{\left(\mathcal{D}\nabla \bU_{h}^{n},\nabla \mathbf{I}\right)_{\Omega}}_{= 0}
	+\underbrace{\left(\bbeta\cdot\nabla\bU_h^n,\mathbf{I}\right)_{\Omega}}_{= 0}
	+\mu\left( \mathbf{U}_{h}^{n},\mathbf{I}\right)_{\Omega}
	+\underbrace{J(\mathbf{U}_{h}^{n},\mathbf{I})}_{= 0}\right)
	=\epsilon\mu( \mathbf{U}_{h}^{n},\mathbf{I})_{\Omega}.
	\end{equation}

	Inserting \eqref{Sta-second} and \eqref{Sta-third} in \eqref{Sta-first} leads to
	\begin{align*}
	\frac{1}{2 \Delta t} \Big(\| \mathbf{U}_{h}^{n} \|_{0,\Omega}^{2}
	&- \| \mathbf{U}_{h}^{n-1} \|_{0,\Omega}^{2}
	+ \| \mathbf{U}_{h}^{n} - \mathbf{U}_{h}^{n-1} \|_{0,\Omega}^{2}\Big)
	+\frac{\epsilon}{\Delta t} (\mathbf{U}_{h}^{n-1} - \mathbf{U}_{h}^{n}, \mathbf{I})_{\Omega}
	+ a_{J}(\mathbf{U}_{h}^{n}, \mathbf{U}_{h}^{n})
	\\
	&\leq (\mathbf{F}^{n}, \mathbf{U}_{h}^{n})_{\Omega}
	-\epsilon \left( (\mathbf{F}^{n}, \mathbf{I})_{\Omega}-\mu( \mathbf{U}_{h}^{n},\mathbf{I})_{\Omega} \right).
	\end{align*}
	Adding from $n=1$ to $n=N$ and using the Cauchy--Schwarz inequality, together with
	\[
	\|\mathbf{I}\|_{0,\Omega}= \sqrt{d} |\Omega|^{1/2},
	\]
	gives
	\begin{align*}
	\| \mathbf{U}_{h}^{N} \|_{0,\Omega}^{2}
	&- \| \mathbf{U}_{h}^{0} \|_{0,\Omega}^{2}
	+ \sum_{n=1}^{N}\| \mathbf{U}_{h}^{n} - \mathbf{U}_{h}^{n-1} \|_{0,\Omega}^{2}
	+  2 \Delta t\sum_{n=1}^{N} a_{J}(\mathbf{U}_{h}^{n}, \mathbf{U}_{h}^{n})
	\\
	&\leq 2 \Delta t \sum_{n=1}^{N}\Big(
	(\epsilon\mu \sqrt{d} \lvert \Omega \rvert^{\frac{1}{2}}+\|\mathbf{F}^{n}\|_{0,\Omega}) \| \mathbf{U}_{h}^{n}\|_{0,\Omega}
	+\epsilon \sqrt{d}\lvert \Omega \rvert^{\frac{1}{2}}\|\mathbf{F}^{n}\|_{0,\Omega}\Big)
	\\
	&\quad
	+2\epsilon \sqrt{d}\lvert \Omega \rvert^{\frac{1}{2}}  \| \mathbf{U}_{h}^{N}\|_{0,\Omega}
	-2\epsilon(\mathbf{U}_{h}^{0} , \mathbf{I})_{\Omega}.
	\end{align*}

        Now, for each $n$, we use
	\[
	2ab \le \frac{1}{2T}a^2 + 2T b^2,
	\]
	with $a=\|\mathbf{U}_h^n\|_{0,\Omega}$ and $b=\epsilon\mu \|\mathbf{I}\|_{0,\Omega}+\|\mathbf{F}^n\|_{0,\Omega}$, and we also use
	\[
	2ab \le \frac12 a^2 + 2 b^2,
	\]
	with $a=\|\mathbf{U}_h^N\|_{0,\Omega}$ and $b=\epsilon\|\mathbf{I}\|_{0,\Omega}$.
	After rearranging terms and multiplying by $2$ we obtain
	\begin{align*}
	\| \mathbf{U}_{h}^{N} \|_{0,\Omega}^{2}
	&+ 2\sum_{n=1}^{N}\| \mathbf{U}_{h}^{n} - \mathbf{U}_{h}^{n-1} \|_{0,\Omega}^{2}
	+ 4 \Delta t\sum_{n=1}^{N} a_{J}(\mathbf{U}_{h}^{n}, \mathbf{U}_{h}^{n})
	\\
	&\leq \frac{\Delta t}{T}\sum_{n=1}^{N}\| \mathbf{U}_{h}^{n}\|_{0,\Omega}^{2}
	+4\Delta t\sum_{n=1}^{N}\Big(
	T(\epsilon\mu \sqrt{d}\lvert \Omega \rvert^{\frac{1}{2}}+\|\mathbf{F}^{n}\|_{0,\Omega})^{2}
	+\epsilon \sqrt{d} \lvert \Omega \rvert^{\frac{1}{2}}\|\mathbf{F}^{n}\|_{0,\Omega}\Big)
	\\
	&\quad
	+4d\epsilon^{2}\lvert \Omega \rvert
	+2\|\mathbf{U}_{h}^{0}\|_{0,\Omega}^{2}
	-4\epsilon\int_\Omega\mathrm{tr}(\mathbf{U}_{h}^{0}) \mathrm{d}\bx.
	\end{align*}

Finally, we apply Lemma~\ref{13} with the choices
\[
	k=\Delta t,\qquad
	a_n=\|\bU_h^n\|_{0,\Omega}^2,\qquad
	\gamma_n=\frac{1}{T},\qquad
	\sigma_n=\Big(1-\frac{\Delta t}{T}\Big)^{-1},
\]
\[
	b_n=\frac{2}{\Delta t}\|\bU_h^n-\bU_h^{n-1}\|_{0,\Omega}^2+4 a_J(\bU_h^n,\bU_h^n),\qquad
	c_n=4\Big(T(\epsilon\mu \sqrt{d}|\Omega|^{1/2}+\|\mathbf F^n\|_{0,\Omega})^2+\epsilon\sqrt{d}|\Omega|^{1/2}\|\mathbf F^n\|_{0,\Omega}\Big),
\]
and
\[
	B=4d\epsilon^{2}\lvert \Omega \rvert
	+2\|\mathbf{U}_{h}^{0}\|_{0,\Omega}^{2}
	-4\epsilon\int_\Omega\mathrm{tr}(\mathbf{U}_{h}^{0}) \mathrm{d}\bx.
\]
Since $\Delta t=T/N$ and $N\ge 2$, we have $k\gamma_n=\Delta t/T=1/N<1$, and hence $\sigma_n=(1-1/N)^{-1}=N/(N-1)$. Therefore
\[
	\exp\Big(k\sum_{n=1}^{N}\sigma_n\gamma_n\Big)
	=\exp\Big(\Delta t\sum_{n=1}^{N}\frac{N}{N-1} \frac{1}{T}\Big)
	=\exp\Big(\frac{N}{N-1}\Big)
	\le e^{2} ,
\]
and Lemma~\ref{13} yields \eqref{Stabilityestimate}.
\end{proof}

The next result states the (energy-norm) error estimate for the method \eqref{eq199}.

\begin{theorem}[A priori error estimate]
  \label{Theorem11}
	Let $k\ge 1$, let $\bU^{0}\in \left(H^{k+1}(\Omega)\right)^{d\times d}$, and let $\bU$ be the solution of \eqref{CDR} satisfying
	\[
	\bU\in  L^{\infty}((0,T);(H^{k+1}(\Omega))^{d\times d})\cap (H^{1}(0,T);(H^{k+1}(\Omega))^{d\times d})\cap H^{2}((0,T);(L^{2}(\Omega))^{d\times d}),
	\]
	with $\bU(\cdot,t)\in (H^{1}_{0}(\Omega))^{d\times d,\textrm{sym}}$ for almost all $t\in [0,T]$.
	Assume in addition that \eqref{eq14} holds for $\bU(t)$ for almost all $t\in[0,T]$, so that $\mathbf I_h\bU^n\in\mathbb V_{\calP}^{\epsilon,\kappa}$ for each $n$.
	Let $\bU_{h}^{n}\in \mathbb{V}_{\calP}$ be the solution of \eqref{eq199} 
	at the time step $n$.
	Defining $\bE^{n}=\bU^{n}_{h}-\bU^{n}$, there exists a constant $C>0$,  independent of $h$ and $\Delta t$ (and depending only on $k$ and mesh shape-regularity), such that
\begin{equation}
	\begin{aligned}
		\Big( \|\bE^N \|_{0,\Omega}^2&+\Delta t\sum_{n=1}^{N}\|\bE^n \|_{a_{J}}^2 \Big)^{\frac{1}{2}}
		\leq C\Bigg[ \Delta t \left\|	\partial_{t}^{2}\mathbf{U} \right\|_{L^{2}((0,T);L^{2}(\Omega))} \\
		&\qquad\qquad+ h^{k}\Big(\big(T\|\bbeta\|_{0,\infty,\Omega}^{2}+\|\mathcal{D}^{\frac{1}{2}}\|_{0,\infty,\Omega}^{2}+\gamma h\|\bbeta\|_{0,\infty,\Omega}+\mu h^{2}\big)^{\frac{1}{2}}
		\left\|	\mathbf{U} \right\|_{L^{2}((0,T);H^{k+1}(\Omega))}\\
			&\qquad\qquad\qquad\qquad\qquad\qquad
			+h\Big(\left\|\partial_{t}\mathbf{U} \right\|_{L^{2}((0,T);H^{k+1}(\Omega))}+ \max\limits_{1\leq n\leq N}|\bU^{n} |_{k+1,\Omega}\Big)\Big)\Bigg].
	\end{aligned}
\end{equation}
\end{theorem}

\begin{proof}
	Let $\mathbf{U}^n = \mathbf{U}(t_n)$ denote the exact solution at time $t_n$.
	Using the test function $\mathbf{I}_h \mathbf{U}^n\in \mathbf{V}_{\calP}^{\epsilon,\kappa}$ in \eqref{eq199} gives
	\begin{align}
		\begin{split}\label{inequalii}
			\left(\delta\mathbf{U}_h^n, \mathbf{I}_h\mathbf{U}^n - \mathbf{U}_h^n\right)_{\Omega}
			+  a_J(\mathbf{U}_h^{n}, \mathbf{I}_h\mathbf{U}^n - \mathbf{U}_h^n)
			\geq (\mathbf{F}^{n}, \mathbf{I}_h\mathbf{U}^n - \mathbf{U}_h^n)_{\Omega}.
		\end{split}
	\end{align}
	Also, setting $\mathbf{V}_h = \mathbf{I}_h \mathbf{U}^n - \mathbf{U}_h^n$ in the continuous problem \eqref{eq82} yields
	\begin{align}
		\begin{split}\label{exact12}
			\left(\partial_t\mathbf{U}(t_{n}), \mathbf{I}_h \mathbf{U}^n - \mathbf{U}_h^n\right)_{\Omega}
			+ a(\mathbf{U}(t_{n}), \mathbf{I}_h \mathbf{U}^n - \mathbf{U}_h^n)
			=(\mathbf{F}(t_{n}), \mathbf{I}_h \mathbf{U}^n - \mathbf{U}_h^n)_{\Omega}.
		\end{split}
	\end{align}

	Now, we decompose $\bE^{n}=\bU^{n}_{h}-\bU^{n}$ as
\begin{align}
	\bE^{n}=\left(\bU^{n}_{h}-\mathbf{I}_h \mathbf{U}^n\right)+\left(\mathbf{I}_{h}\bU^{n}-\bU^{n}\right)=:\bE^{n}_{h}+\eta^{n}_{h}.\label{error1}
\end{align}
Since $k\ge 1$ and $\bU^{n}\in (H^{k+1}(\Omega))^{d\times d}\subset (H^{2}(\Omega))^{d\times d}$, we have $\llbracket \nabla \bU^{n}\rrbracket=\mathbf 0$ on every interior facet and hence
\[
	J(\bU^{n},\bV_h)=0\qquad \forall \bV_h\in\mathbb V_{\calP}.
\]
Consequently,
\begin{align}
	J(\bU^{n}_{h}, \bV_{h})
	=J(\bU^{n}_{h}- \bU^{n}, \bV_{h})
	=J( \bE^{n}_{h}, \bV_{h})+J(\eta^{n}_{h}, \bV_{h}).\label{CIP12}
\end{align}
Adding $J(\bU^n,\mathbf I_h\bU^n-\bU_h^n)=0$ to \eqref{exact12}, we may equivalently replace $a(\bU^n,\cdot)$ by $a_J(\bU^n,\cdot)$ in \eqref{exact12}.
Therefore, subtracting \eqref{exact12} from \eqref{inequalii} gives
\begin{align}\label{subtraction}
	\Big(\delta\mathbf{U}_h^n - \partial_t\mathbf{U}(t_{n}), \mathbf{I}_h \mathbf{U}^n - \mathbf{U}_h^n\Big)_{\Omega}
	+ a_J(\bU_h^{n}-\bU(t_{n}), \mathbf{I}_h \mathbf{U}^n - \mathbf{U}_h^n)
	\geq 0.
\end{align}

The first term in \eqref{subtraction} may be rewritten as
\begin{align}
	\delta \bU^{n}_{h}-\partial_{t} \bU(t_{n})
	&=\big(\delta \bU^{n}_{h}- \delta(\mathbf{I}_{h}\bU^{n})\big)
	+\big(\delta\bU^{n}-\partial_{t} \bU(t_{n})\big)
	+\big(\delta(\mathbf{I}_{h}\bU^{n})-\delta\bU^{n}\big)\nonumber\\
	&= \delta \bE^{n}_{h}+\big(\delta\bU^{n}-\partial_{t} \bU(t_{n})\big)+\delta \eta^{n}_{h}. \label{time12}
\end{align}
Using \eqref{time12} and $\mathbf I_h\bU^n-\bU_h^n=-\bE_h^n$, \eqref{subtraction} becomes
\begin{align}\label{Eq12367443}
	\big(\delta \bE^{n}_{h}, -\bE^{n}_{h}\big)_{\Omega}+a_{J}(\bE^{n}_{h},-\bE^{n}_{h})
	\geq
	(\partial_{t} \bU(t_{n})-\delta \bU^{n},-\bE^{n}_{h})_{\Omega}
	-(\delta \eta^{n}_{h}, -\bE^{n}_{h})_{\Omega}
	-a_{J}( \eta^{n}_{h}, -\bE^{n}_{h}).
\end{align}
Multiplying \eqref{Eq12367443} by $-1$ gives the equivalent form
\begin{align}\label{Eq12367443b}
	\big(\delta \bE^{n}_{h}, \bE^{n}_{h}\big)_{\Omega}+a_{J}(\bE^{n}_{h},\bE^{n}_{h})
	\leq
	(\delta \bU^{n}-\partial_{t} \bU(t_{n}),\bE^{n}_{h})_{\Omega}
	+(\delta \eta^{n}_{h}, \bE^{n}_{h})_{\Omega}
	+a_{J}( \eta^{n}_{h}, \bE^{n}_{h}).
\end{align}

Using the identity $2(a^{2}-ab)=a^{2}-b^{2}+(a-b)^{2}$ in $(\delta\bE_h^n,\bE_h^n)_\Omega$ and the Cauchy--Schwarz inequality on the right-hand side of \eqref{Eq12367443b}, we obtain
\begin{align*}
	&\frac{1}{2\Delta t}\Big(\|\bE_h^n \|_{0,\Omega}^2-\| \bE_h^{n-1}\|_{0,\Omega}^2+\|\bE_h^n - \bE_h^{n-1}\|_{0,\Omega}^2\Big)
	+ \|\bE_h^n\|_{a_J}^2\\
	\leq&\, \Big(\|\delta \bU^{n}-\partial_{t}\bU(t_n)\|_{0,\Omega}+\|\delta\eta_h^n\|_{0,\Omega}
	+\|\bbeta\|_{0,\infty,\Omega} |\eta_h^n|_{1,\Omega}\Big)\, \|\bE_h^n\|_{0,\Omega} + \|\eta_h^n\|_{a_J} \|\bE_h^n\|_{a_J}.
\end{align*}
Applying Young's inequality yields
\begin{align*}
	&\|\bE_h^n \|_{0,\Omega}^2-\| \bE_h^{n-1}\|_{0,\Omega}^2+\|\bE_h^n - \bE_h^{n-1}\|_{0,\Omega}^2
	+ \Delta t \|\bE_h^n\|_{a_J}^{2}\\
	\leq&\,  2\Delta t\Big(T(\|\delta \bU^{n}-\partial_{t}\bU(t_n)\|_{0,\Omega}+\|\delta\eta_h^n\|_{0,\Omega}
	+\|\bbeta\|_{0,\infty,\Omega} |\eta_h^n|_{1,\Omega})^2+\frac{1}{T}\|\bE_h^n\|_{0,\Omega}^2\Big)
	+\Delta t \| \eta_h^n\|_{a_J}^{2}.
\end{align*}

Summing from $n=1$ to $n=N$ and using $\bE_h^{0}=\bU_h^0-\mathbf I_h\bU^0=\mathbf 0$ gives
\begin{equation}
	\begin{aligned}\label{error432}
	&\|\bE_h^N \|_{0,\Omega}^2
	+\sum_{n=1}^{N}\|\bE_h^n - \bE_h^{n-1}\|_{0,\Omega}^2
	+ \Delta t \sum_{n=1}^{N}\|\bE_h^n \|_{a_{J}}^{2}\\
	\leq&\, C\Delta t\sum_{n=1}^{N}\bigg(
	T\Big(\|\delta \bU^{n}-\partial_{t}\bU(t_n)\|_{0,\Omega}^{2}
	+\|\delta\eta_h^n\|_{0,\Omega}^{2}
	+\|\bbeta\|_{0,\infty,\Omega}^{2} |\eta_h^n|_{1,\Omega}^{2}\Big)
	+\|\eta_h^n\|_{a_J}^{2}
	+\frac{1}{T}\|\bE_h^n\|_{0,\Omega}^{2}\bigg).
	\end{aligned}
\end{equation}

We now apply Gr\"{o}nwall's Lemma~\ref{13} with the choices
\[
	k=\Delta t,\qquad
	a_n=\|\bE_h^n\|_{0,\Omega}^2,\qquad
	\gamma_n=\frac{1}{T},\qquad
	\sigma_n=\Big(1-\frac{\Delta t}{T}\Big)^{-1},
\]
\[
	b_n=\frac{1}{\Delta t}\|\bE_h^n-\bE_h^{n-1}\|_{0,\Omega}^2+\|\bE_h^n\|_{a_J}^{2},
	\]
\[
	c_n=C\Big(T(\|\delta \bU^{n}-\partial_{t}\bU(t_n)\|_{0,\Omega}^{2}
	+\|\delta\eta_h^n\|_{0,\Omega}^{2}
	+\|\bbeta\|_{0,\infty,\Omega}^{2} |\eta_h^n|_{1,\Omega}^{2})
	+\|\eta_h^n\|_{a_J}^{2}\Big),
\]
and $B=0$. Since $\Delta t=T/N$ and $N\ge 2$, we have $k\gamma_n=\Delta t/T=1/N<1$ and hence
\[
	\exp\Big(k\sum_{n=1}^{N}\sigma_n\gamma_n\Big)
	=\exp\Big(\frac{N}{N-1}\Big)\le e^{2}.
\]
Therefore,
\begin{align}
	&\|\bE_h^N \|_{0,\Omega}^2+\Delta t \sum_{n=1}^{N}\|\bE_h^n \|_{a_{J}}^{2} \nonumber\\
	\leq&\,  C e^{2} \Delta t\sum_{n=1}^{N}\Big(
	T\|\delta \bU^{n}-\partial_{t}\bU(t_n)\|_{0,\Omega}^{2}
	+T\|\delta\eta_h^n\|_{0,\Omega}^{2}
	+T\|\bbeta\|_{0,\infty,\Omega}^{2} |\eta_h^n|_{1,\Omega}^{2}
	+\|\eta_h^n\|_{a_J}^{2}\Big).\label{53last}
\end{align}

For the time truncation term, Taylor's theorem gives, for each $n$,
\[
	\partial_t \bU(t_n)-\delta \bU^n
	=\frac{1}{\Delta t}\int_{t_{n-1}}^{t_n}(t_n-s) \partial_t^2\bU(s) {\rm d}s,
\]
and hence, by Cauchy--Schwarz,
\begin{equation}
	\sum_{n=1}^{N} \left\|\partial_t\mathbf{U}(t_{n})-\delta \mathbf{U}^n\right\|_{0,\Omega}^2
	\leq C\Delta t\int_{0}^{T} \left\|\partial_{t}^{2}\mathbf{U}(t)\right\|_{0,\Omega}^2 {\rm d}t.  \label{time_trunc}
\end{equation}

Next, using the tensor Lagrange approximation \eqref{tensorlagrange}, we have
\begin{align}
	|\eta^{n}_{h}|_{1,\Omega}^{2}\leq Ch^{2k}|\bU^{n} |_{k+1,\Omega}^{2},
	\qquad
	\|\eta^n_h\|_{0,\Omega}^2\le C  h^{2k+2}|\bU^n|_{k+1,\Omega}^{2}.\label{inequal32}
\end{align}
Moreover,
\begin{equation}
	\begin{aligned}
		\sum_{n=1}^{N}\|\delta \eta_{h}^{n}\|_{0,\Omega}^{2}
		&=\sum_{n=1}^{N}\Big\|\delta (\mathbf{I}_h\bU^{n})-\delta(\bU^{n})\Big\|_{0,\Omega}^{2}
		\leq Ch^{2k+2}\sum_{n=1}^{N}|\delta \mathbf{U}^{n} |_{k+1,\Omega}^{2}\\
		&\leq C h^{2k+2}\sum_{n=1}^{N}\left|\frac{1}{\Delta t}\int_{t_{n-1}}^{t_{n}}\partial_{t} \mathbf{U}(t) {\rm d}t\right|_{k+1,\Omega}^{2}
		\leq C\frac{h^{2k+2}}{\Delta t}\int_{0}^{T}|\partial_{t}\mathbf{U}(t)|_{k+1,\Omega}^{2}{\rm d}t.
	\end{aligned}
\end{equation}

Finally, for the stabilisation part we use the discrete trace inequality \eqref{Tensortrace} (applied componentwise) together with standard scaling and the interpolation estimates \eqref{tensorlagrange} to obtain
\begin{align}
	J(\eta_{h}^{n},\eta_{h}^{n})
	= \gamma \sum_{F\in\calF_I^{}}\int_F\|\bbeta \|_{0,\infty,F}  h_{F}^{2} \llbracket\nabla \eta_{h}^{n} \rrbracket :\llbracket\nabla \eta_{h}^{n} \rrbracket  {\rm d}s
	\leq C\gamma h^{2k+1}\|\bbeta\|_{0,\infty,\Omega} |\bU^{n} |_{k+1,\Omega}^{2}.\label{JJJ}
\end{align}
Using \eqref{inequal32} and \eqref{JJJ}, we obtain
\begin{equation}
	\begin{aligned}
		T\|\bbeta\|_{0,\infty,\Omega}^{2} |\eta^{n}_{h}|_{1,\Omega}^{2}
		+\|\eta^{n}_{h}\|_{a_J}^{2}
		&\leq Ch^{2k}\Big(\big(T\|\bbeta\|_{0,\infty,\Omega}^{2}
		+\|\mathcal{D}^{\frac{1}{2}}\|_{0,\infty,\Omega}^{2}
		+\gamma h\|\bbeta\|_{0,\infty,\Omega}
		+\mu h^{2}\big) |\mathbf{U}^{n}|_{k+1,\Omega}^{2}\Big).
	\end{aligned}
\end{equation}
Summing in $n$ and using Cauchy--Schwarz yields
\begin{align}\label{eq:summed_interp_bound}
		&\sum_{n=1}^{N}\Big(T\|\bbeta\|_{0,\infty,\Omega}^{2} |\eta^{n}_{h}|_{1,\Omega}^{2}
		+\|\eta^{n}_{h}\|_{a_J}^{2}\Big) \nonumber\\
		\leq&\,  \frac{Ch^{2k}}{\Delta t}\big(T\|\bbeta\|_{0,\infty,\Omega}^{2}
		+\|\mathcal{D}^{\frac{1}{2}}\|_{0,\infty,\Omega}^{2}
		+\gamma h\|\bbeta\|_{0,\infty,\Omega}
		+\mu h^{2}\big)\int_{0}^{T}|\mathbf{U}(t)|_{k+1,\Omega}^{2}{\rm d}t.
\end{align}

Inserting \eqref{time_trunc}, \eqref{eq:summed_interp_bound} and the bound for $\sum_n\|\delta\eta_h^n\|_{0,\Omega}^2$ into \eqref{53last}, and rearranging, gives
\begin{equation}
	\begin{aligned}
		\max_{1\leq n\leq N}\|\bE_h^n \|_{0,\Omega}^2+\Delta t\sum_{n=1}^{N}\|\bE_h^n \|_{a_{J}}^2
		\leq C\Bigg[
		\Delta t^{2}\int_{0}^{T} \left\|\partial_{t}^{2}\mathbf{U}\right\|_{0,\Omega}^2 {\rm d}t\\
		\qquad\qquad
		+h^{2k}\Big(\big(T\|\bbeta\|_{0,\infty,\Omega}^{2}+\|\mathcal{D}^{\frac{1}{2}}\|_{0,\infty,\Omega}^{2}+\gamma h\|\bbeta\|_{0,\infty,\Omega}+\mu h^{2}\big)\int_{0}^{T}|\mathbf{U}(t)|_{k+1,\Omega}^{2}{\rm d}t\\
		\qquad\qquad\qquad\qquad
		+h^{2}\int_{0}^{T}|\partial_{t}\mathbf{U}(t)|_{k+1,\Omega}^{2}{\rm d}t\Big)\Bigg].
	\end{aligned}
\end{equation}

Finally, using the triangle inequality,
\[
	\|\bE^n\|_{0,\Omega} \le \|\bE_h^n\|_{0,\Omega}+\|\eta_h^n\|_{0,\Omega},
	\qquad
	\|\bE^n\|_{a_J} \le \|\bE_h^n\|_{a_J}+\|\eta_h^n\|_{a_J},
\]
together with \eqref{inequal32}, yields the stated estimate after taking square roots.
\end{proof}

\section{Numerical experiments}\label{sec:numerics}

We implement the nodally bound-preserving scheme \eqref{eq199} using the FEniCS software framework \cite{fenicsx2023, dolfinx2023}. The variational inequality is resolved at each time step through an iterative procedure based on the approach presented in \cite{Kor76}.  Alternative solvers for closely related nodally constrained schemes include active-set strategies and projection-based fixed-point variants, see \cite{ashby2025nodally,barrenechea2025nodally}.
We denote by $\mathbb{A}$ and $\mathbf{L}$ the assembled finite element stiffness matrix and load vector of the problem \eqref{variational} at the time step $n$, respectively, and introduce a parameter $\omega > 0$. Let $\mathcal{P}$ be the (nonlinear) nodal projection operator $\mathbb{V}_\calP^{}\to \mathbb{V}_{\calP}^{\epsilon,\kappa}$ defined by \eqref{posoperator}.

Denoting by $\mathbf{\bU}^{n}$ the vector of degrees of freedom of the solution of the problem at time $n$, the iterative solution of \eqref{variational} takes as initial condition
$(\mathbf{\bU}^{n})^0=\mathbf{\bU}^{n-1}$, and then, for $r=1,2,\cdots$, computes the iteration
\begin{equation}
\begin{aligned}\label{iterative12}
	(\mathbf{\bV}^{n})^r &= \mathcal{P}\left((\mathbf{\bU}^{n})^{r-1} - \omega\left(\mathbb{A} (\mathbf{\bU}^{n})^{r-1} - \mathbf{L}\right)\right),  \\
	(\mathbf{\bU}^{n})^r &= \mathcal{P}\left((\mathbf{\bU}^{n})^{r-1} - \omega\left(\mathbb{A} (\mathbf{\bV}^{n})^r - \mathbf{L}\right)\right).
\end{aligned}
\end{equation}
The iterations continue until $\|(\mathbf{\bU}^{n})^r - (\mathbf{\bU}^{n})^{r-1}\|_{\ell^2} < \text{tol} := 10^{-6}$. 

Our computational setup consists of two-dimensional domains discretised using triangular meshes. In the following examples, $P-1$ indicates the number of divisions in the $x$ and
$y$ directions, resulting in a total of $P^2$ vertices, including the
boundary. 

We test both the original method \eqref{eq199},  referred to as (BP-Euler), and also include the following Crank--Nicolson
version: we set
\begin{equation*}
 t_{n-\frac{1}{2}} = \frac{t_{n}+t_{n-1}}{2}\quad,\quad 
 \mathbf{U}_{h}^{n-\frac{1}{2}} := (\mathbf{U}_{h}^{n} +\mathbf{U}_{h}^{n-1})/2,
\end{equation*}
the Crank--Nicolson version of \eqref{eq199} reads:
\begin{equation}
	\left\{ \begin{array}{ll}
		\text{For $1\leq n \leq N$, find $\mathbf{U}_h^n \in \mathbb{V}_{\calP}^{\epsilon,\kappa}$ such that}\vspace{.1cm}
		\\  
		(\delta (\mathbf{U}_{h}^{n}),\mathbf{V}_{h}-\mathbf{U}_{h}^{n})_{\Omega} + a_{J}(\mathbf{U}_{h}^{n-\frac{1}{2}},\mathbf{V}_{h}-\mathbf{U}_{h}^{n}) \geq (\mathbf{F}^{n-\frac{1}{2}},\mathbf{V}_{h}-\mathbf{U}_{h}^{n})_{\Omega} \hspace{1cm} \forall \mathbf{V}_{h}\in \mathbb{V}_{\calP}^{\epsilon,\kappa}, \vspace{.1cm}\\
		\hspace{4.8cm} \mathbf{U}_{h}^{0} = \mathbf{I}_{h}\mathbf{U}^{0}.
	\end{array} \right.\label{eq199CN}
\end{equation}
We shall refer to this last variant as the bound-preserving Crank--Nicolson method (BP-CN), and will report its results even if its stability and convergence have not been established.

\begin{example}({\bf Circular convection})
In this example we take $d=3$.
We first test the performance of the present method in the extension to a time-dependent tensor-valued setting (introduced in~\cite{lohmann2019algebraic}) of the
stationary circular convection problem~\cite{hubbard2007non}. The partial differential equation at hand is:
 	\begin{equation}
 		\begin{split}
 			\left\{ \begin{array}{ll}  
 				\partial_{t}\mathbf{U}+\mathrm{div}(\bbeta \mathbf{U})& = 0  \hspace{1cm} \text{in } (0,T]\times\Omega,\quad \Omega = (0,1)^2,\\
 				\hspace{1.6cm} \mathbf{U}(\cdot,0) &= \mathbf{I}\ \text{in } \Omega,\\
 				\hspace{2.3cm}\mathbf{U} &= \mathbf{U}_{\text{in}} \quad \text{on } (0,T]\times \Gamma_{\text{in}} = (0,T]\times\big([0,1] \times \{0\} \cup \{1\} \times [0,1]\big),
 			\end{array} \right.\label{example22}
 		\end{split}
 	\end{equation}
 	where $\bbeta = (-y, x)^{T}$, and so, the inflow boundary
 	$\Gamma_{\text{in}} := \{ \bx\in \partial \Omega : \bbeta(\bx) \cdot \mathbf{n}(\bx) < 0 \}$ is indeed the one specified in \eqref{example22}.
\end{example}

The inflow boundary condition \(\mathbf{U}_{\text{in}} : \Gamma_{\text{in}} \to \mathbb{S}_d\) is transported around the vortex centre,
located at the lower-left corner of the domain, \( \boldsymbol{x}^* = (0,0)^{T}\) (see Figure~\ref{Convectionfig}).
Consequently, the solution \(\mathbf{U} : \Omega \to \mathbb{S}_d\) depends solely on the radial distance
\[
r = \| \boldsymbol{x} - \boldsymbol{x}^* \|_2 = \|  \boldsymbol{x} \|_2,
\]
from the origin, and is uniquely determined by the inflow boundary condition \(\mathbf{U}_{\text{in}}\).
\begin{figure}[h!]
	\centering
	\includegraphics[width=0.4\textwidth]{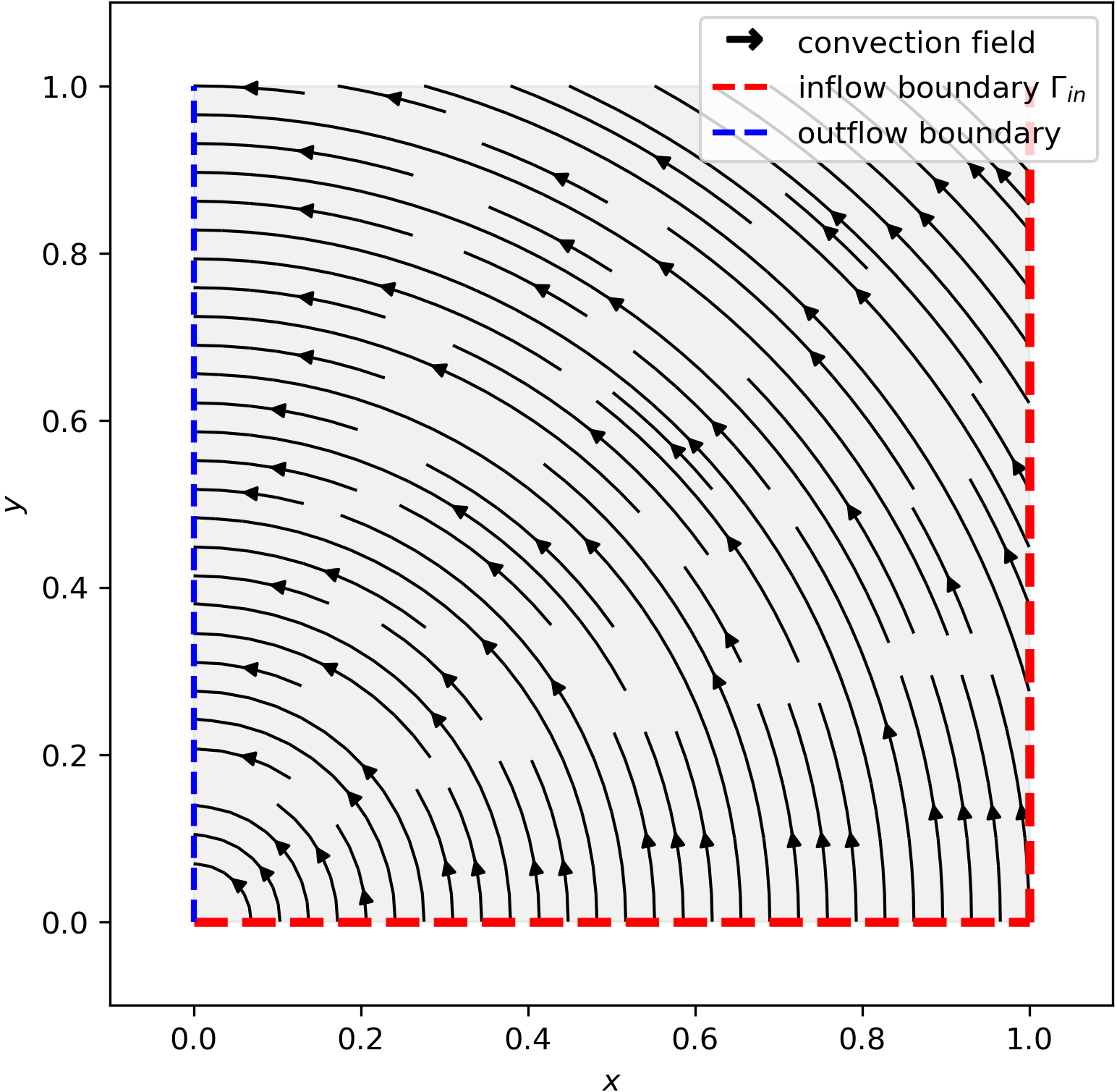}
	\caption{Domain $\Omega$ with velocity field $\bbeta = (-y, x)^{T}$ and inflow boundary $\Gamma_{\text{in}}$.}\label{Convectionfig}
\end{figure}

In the following, we consider two boundary conditions. The first one is the following smooth function,
built with the aim of validating the theoretical results presented in Theorem~\ref{Theorem11}:
\[
\mathbf{U} =
\begin{pmatrix}
	\sin \tilde{r} & \cos \tilde{r} & 0 \\
	\cos \tilde{r} & -\sin \tilde{r} & 0 \\
	0 & 0 & 1
\end{pmatrix}
\begin{pmatrix}
	\sin \tilde{r} & 0 & 0 \\
	0 & 1-\sin \tilde{r} & 0 \\
	0 & 0 & 0
\end{pmatrix}
\begin{pmatrix}
	\sin \tilde{r} & \cos \tilde{r} & 0 \\
	\cos \tilde{r} & -\sin \tilde{r} & 0 \\
	0 & 0 & 1
\end{pmatrix},
\qquad
\tilde{r} := \frac{3}{4} \pi r.
\]
Here, we employ $\gamma = 0.1$ in~\eqref{eq10}, and
in the iterative method \eqref{iterative12} we use $\omega=10^{-4}$.

The convergence behaviour of BP--Euler is illustrated in Figure~\ref{figconvergence}.
In Figure~\ref{figconvergence}\subref{figconvergence1}, \(\|E^{N}\|_{0,\Omega}\) denotes the norm of the difference between the exact solution
and the computed solution at the final time step. In all experiments we set \(T = 4\).
As expected, and as proved in Theorem~\ref{Theorem11}, the Euler method exhibits first-order convergence in time.

Figure~\ref{figconvergence}\subref{figconvergence2} presents the convergence with respect to mesh refinement for the BP--Euler method using $\mathbb{P}_{1}$ and $\mathbb{P}_{2}$ elements.
We observe a convergence of approximately order $1.5$ when using $\mathbb{P}_{1}$ elements and order $2.5$ when using $\mathbb{P}_{2}$ elements.
This shows that the current method also provides the $O(h^{k+1/2})$ convergence proven in \cite{JP86} for a DG discretisation of a scalar transport equation.
The proof of this fact for the CIP stabilisation is an open question, whereas an $O(h^{k+1/2})$ result for a related SUPG setting has been reported in \cite{ashby2025nodally} in the steady-state case.

\begin{figure}[H]
	\centering
	\subfloat[]{
		\includegraphics[width=0.48\textwidth]{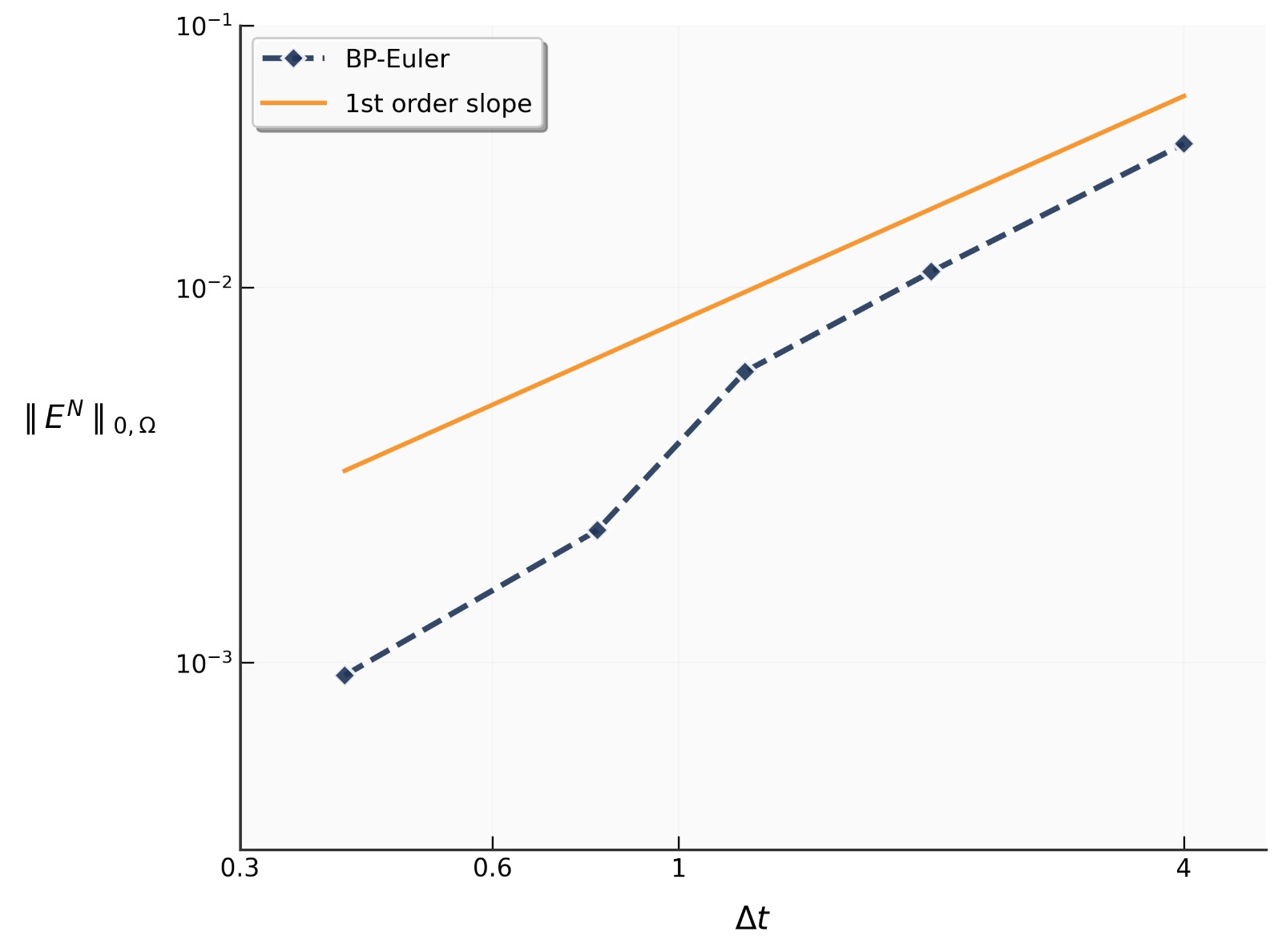}
		\label{figconvergence1}
	}
	\subfloat[]{
		\includegraphics[width=0.48\textwidth]{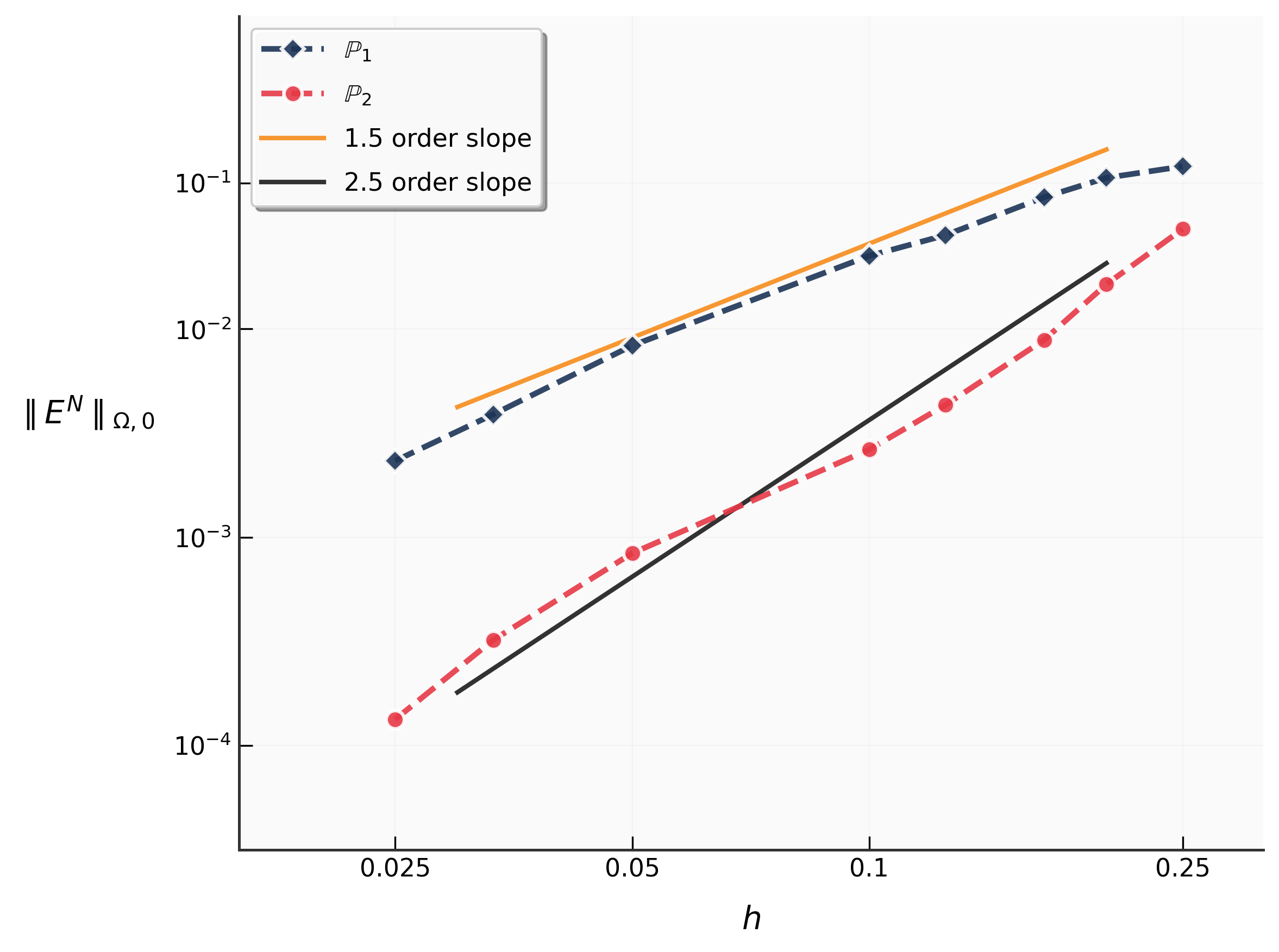}
		\label{figconvergence2}
	}
	\caption{\textbf{Left:} Convergence with respect to the time step size $\Delta t$ for $\mathbb{P}_{1}$ elements with $h=1/50$.
	\textbf{Right:} Convergence with respect to the mesh size $h$ for the BP--Euler method using $\mathbb{P}_{1}$ and $\mathbb{P}_{2}$ elements with $\Delta t = 1/500$.}
	\label{figconvergence}
\end{figure}
 	
\subsubsection{Discontinuous solution}
Next, to test the performance of the method for a problem with sharp fronts, we consider the following discontinuous inflow boundary condition:
\[
\mathbf{U}_{\mathrm{in}} =
\begin{cases}
	\mathbf{U}_1, & 0 < r < \frac{1}{2}, \\
	\mathbf{U}_2, & \frac{1}{2} \leq r < \frac{2}{3}, \\
	\mathbf{U}_3, & \frac{2}{3} \leq r < \frac{3}{4}, \\
	\mathbf{U}_4, & \frac{3}{4} \leq r < \frac{4}{5}, \\
	\mathbf{U}_5, & \frac{4}{5} \leq r < \sqrt{2},
\end{cases}
\]
where the constant parts and their eigenvalue decompositions are given by
\[
\mathbf{U}_1 = \tfrac{1}{3}
\begin{pmatrix}
	1 & 0 & 0 \\
	0 & 1 & 0 \\
	0 & 0 & 1
\end{pmatrix},
\]
\[
\mathbf{U}_2 = \tfrac{1}{75}
\begin{pmatrix}
	32 & 24 & 0 \\
	24 & 18 & 0 \\
	0 & 0 & 25
\end{pmatrix}
= \tfrac{1}{5}
\begin{pmatrix}
	4 & 3 & 0 \\
	3 & -4 & 0 \\
	0 & 0 & 5
\end{pmatrix}
\begin{pmatrix}
	\tfrac{2}{3} & 0 & 0 \\
	0 & 0 & 0 \\
	0 & 0 & \tfrac{1}{3}
\end{pmatrix}
\tfrac{1}{5}
\begin{pmatrix}
	4 & 3 & 0 \\
	3 & -4 & 0 \\
	0 & 0 & 5
\end{pmatrix},
\]
\[
\mathbf{U}_3 = \tfrac{1}{3}
\begin{pmatrix}
	1 & -1 & 0 \\
	-1 & 1 & 0 \\
	0 & 0 & 1
\end{pmatrix}
= \tfrac{1}{2}
\begin{pmatrix}
	\sqrt{2} & \sqrt{2} & 0 \\
	\sqrt{2} & -\sqrt{2} & 0 \\
	0 & 0 & 2
\end{pmatrix}
\begin{pmatrix}
	0 & 0 & 0 \\
	0 & \tfrac{2}{3} & 0 \\
	0 & 0 & \tfrac{1}{3}
\end{pmatrix}
\tfrac{1}{2}
\begin{pmatrix}
	\sqrt{2} & \sqrt{2} & 0 \\
	\sqrt{2} & -\sqrt{2} & 0 \\
	0 & 0 & 2
\end{pmatrix},
\]
\[
\mathbf{U}_4 = \tfrac{1}{3}
\begin{pmatrix}
	1 & 1 & 0 \\
	1 & 1 & 0 \\
	0 & 0 & 1
\end{pmatrix}
= \tfrac{1}{2}
\begin{pmatrix}
	\sqrt{2} & \sqrt{2} & 0 \\
	\sqrt{2} & -\sqrt{2} & 0 \\
	0 & 0 & 2
\end{pmatrix}
\begin{pmatrix}
	\tfrac{2}{3} & 0 & 0 \\
	0 & 0 & 0 \\
	0 & 0 & \tfrac{1}{3}
\end{pmatrix}
\tfrac{1}{2}
\begin{pmatrix}
	\sqrt{2} & \sqrt{2} & 0 \\
	\sqrt{2} & -\sqrt{2} & 0 \\
	0 & 0 & 2
\end{pmatrix},
\]
\[
\mathbf{U}_5 = \tfrac{1}{3}
\begin{pmatrix}
	1 & 1 & 1 \\
	1 & 1 & 1 \\
	1 & 1 & 1
\end{pmatrix}
= \tfrac{\sqrt{6}}{6}
\begin{pmatrix}
	\sqrt{2} & \sqrt{3} & 1 \\
	\sqrt{2} & -\sqrt{3} & 0 \\
	\sqrt{2} & 0 & -2
\end{pmatrix}
\begin{pmatrix}
	1 & 0 & 0 \\
	0 & 0 & 0 \\
	0 & 0 & 0
\end{pmatrix}
\tfrac{\sqrt{6}}{6}
\begin{pmatrix}
	\sqrt{2} & \sqrt{3} & 1 \\
	\sqrt{2} & -\sqrt{3} & 0 \\
	\sqrt{2} & 0 & -2
\end{pmatrix}.
\]
The data are constructed such that the trace is identically equal to $1$ on the inflow boundary $\Gamma_{\text{in}}$, and the eigenvalues remain bounded between $0$ and $1$. Owing to the solenoidal velocity field, these properties are preserved throughout the domain, and thus we take $\epsilon=0$ and $\kappa=1$.
The discontinuities are designed to illustrate the capability of numerical methods to handle different scenarios (see~\cite{lohmann2019algebraic} for more details).

We report results for four schemes: CIP--Euler (the unconstrained baseline), BP--Euler (the bound-preserving method presented in this work), and the analogous Crank--Nicolson variants CIP--CN and BP--CN.
Figure~\ref{fig1} shows the minimal and maximal eigenvalues obtained with the CIP--Euler and BP--Euler methods using $\mathbb{P}_{1}$ elements, together with cross-sections of these eigenvalues taken along the line $y=x$.
In these experiments, we set $P=121$ and $\gamma=10^{-3}$ for the stabilisation term \eqref{eq10},
and used a time step of $\Delta t = 10^{-3}$ and $T=4$. Moreover, we choose $\omega =10^{-2}$ in the iterations of \eqref{iterative12}, which ensures very fast convergence at each time step, reaching convergence after only a few iterations.

For the CIP--Euler solution, both the minimal and maximal eigenvalues violate the bounds given by the inflow data, whereas the BP--Euler method preserves the bounds.
A similar behaviour is observed for the CIP--CN and BP--CN schemes, and the corresponding $\mathbb{P}_{1}$ results are presented in Figure~\ref{fig22}.
Finally, Figure~\ref{fig33} reports the four schemes using quadratic elements ($k=2$), together with the corresponding cross-sections along $y=x$.
The same conclusions hold for $\mathbb{P}_{2}$ elements.  In particular, although the admissibility constraints are imposed only at the nodes, the plotted eigenvalues remain within $[0,1]$ also between the nodes in these tests.

\begin{figure}[H]
	\centering
	\subfloat[Minimal eigenvalue for the CIP--Euler solution.]{
		\includegraphics[width=0.45\textwidth]{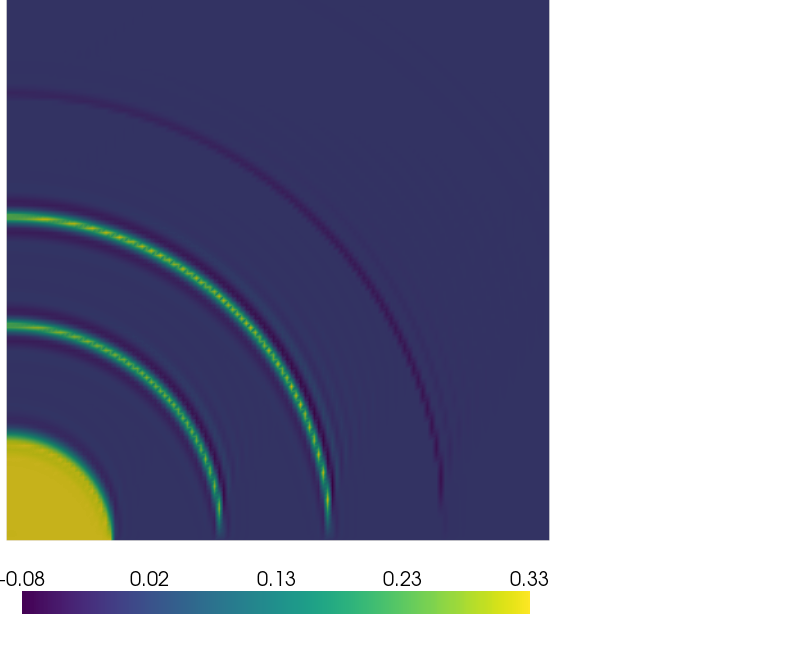}
	}
	\subfloat[Maximal eigenvalue for the CIP--Euler solution.]{
		\includegraphics[width=0.45\textwidth]{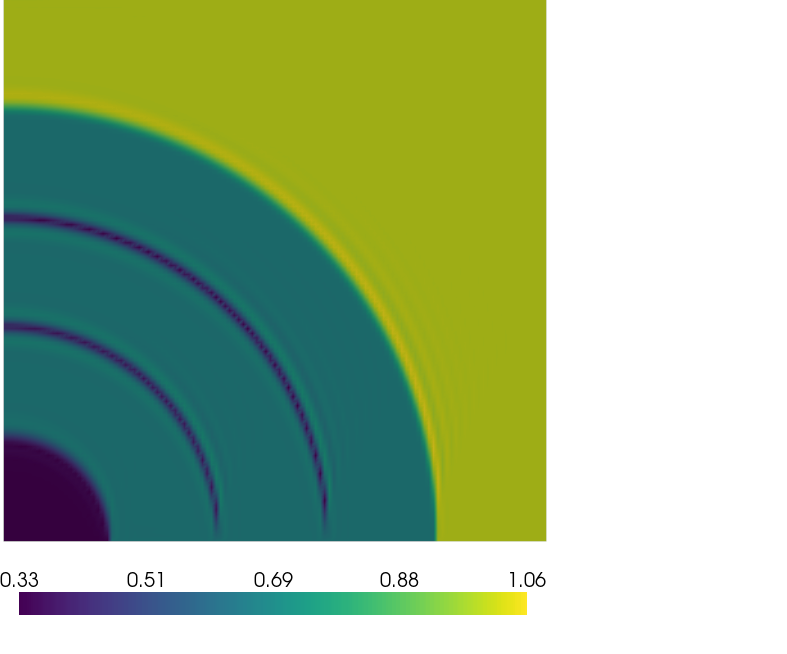}
	}\\
	\subfloat[Minimal eigenvalue for the BP--Euler solution.]{
		\includegraphics[width=0.45\textwidth]{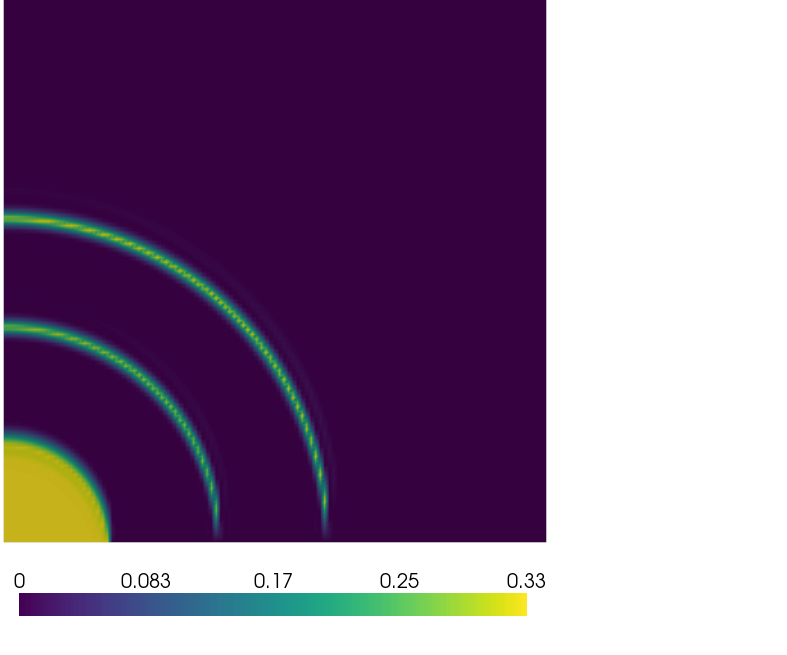}
	}
	\subfloat[Maximal eigenvalue for the BP--Euler solution.]{
		\includegraphics[width=0.45\textwidth]{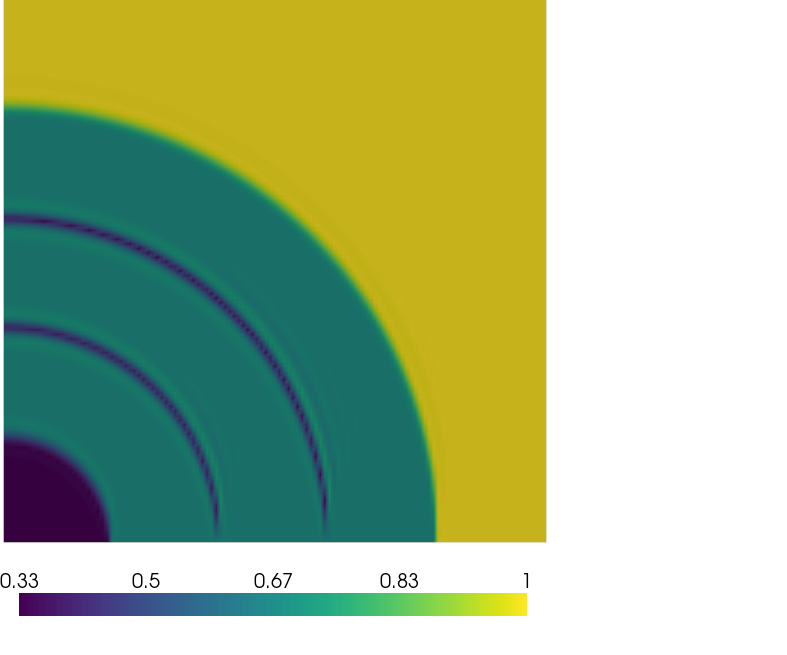}
	}\\
	\subfloat[Cross-section of the minimal eigenvalue along $y=x$.]{
		\includegraphics[width=0.45\textwidth]{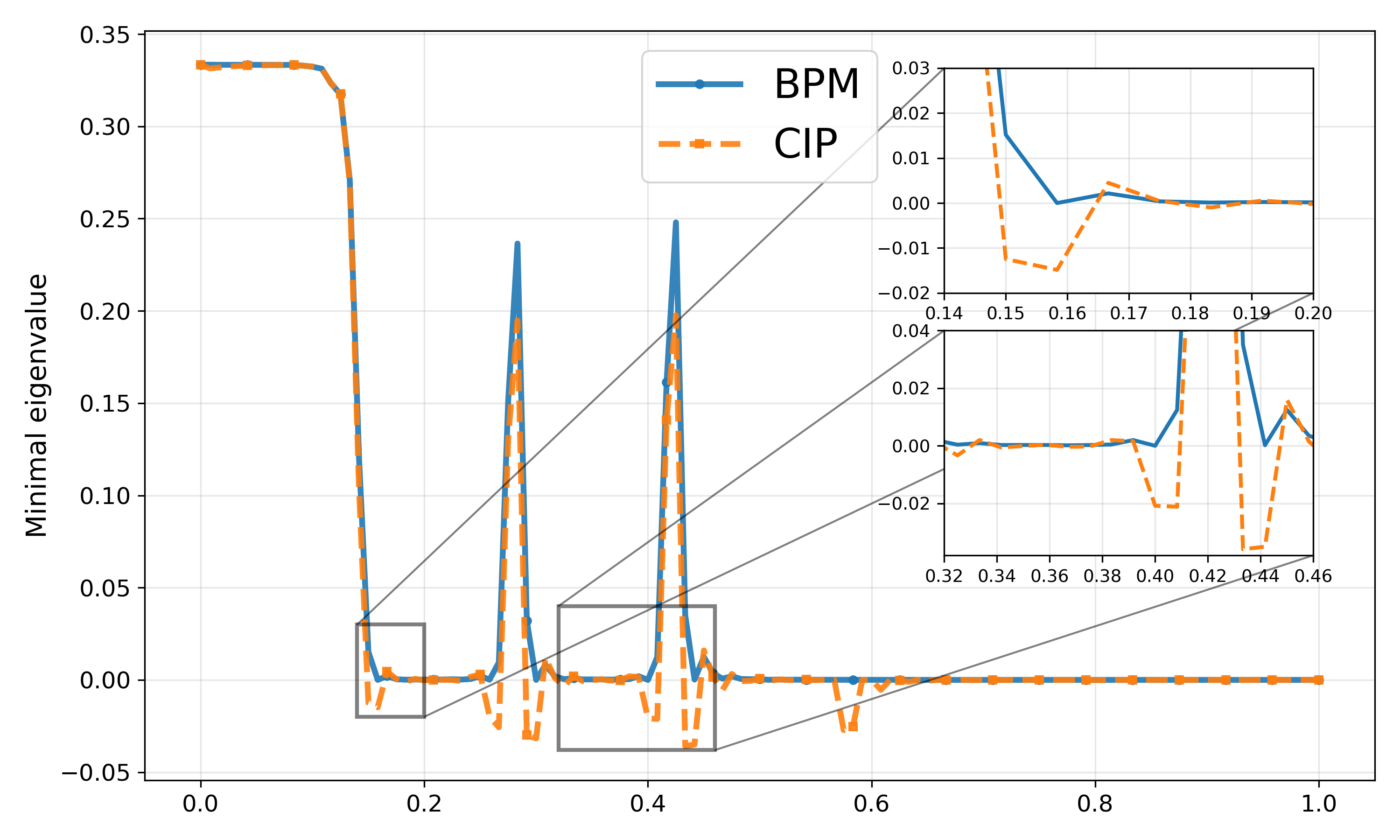}
	}
	\subfloat[Cross-section of the maximal eigenvalue along $y=x$.]{
		\includegraphics[width=0.45\textwidth]{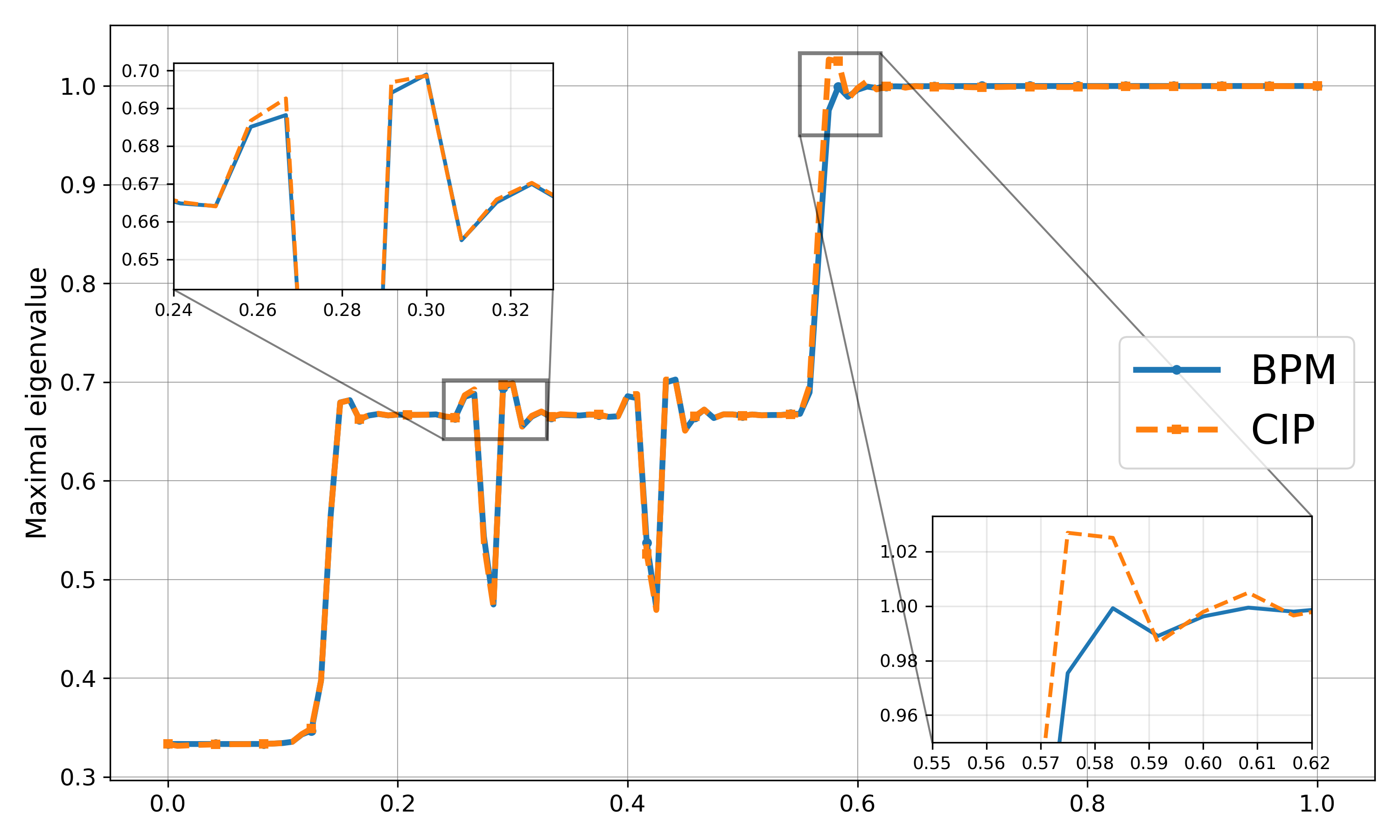}
	}\\
	\caption{Solution of \eqref{example22} at $T=4$ with discontinuous inflow data: minimal and maximal eigenvalues obtained with the CIP--Euler and BP--Euler schemes, together with cross-sections taken along $y=x$ ($\mathbb{P}_{1}$ elements).}
	\label{fig1}
\end{figure}

\begin{figure}[H]
	\centering
	\subfloat[Minimal eigenvalue for the CIP--CN solution.]{
		\includegraphics[width=0.45\textwidth]{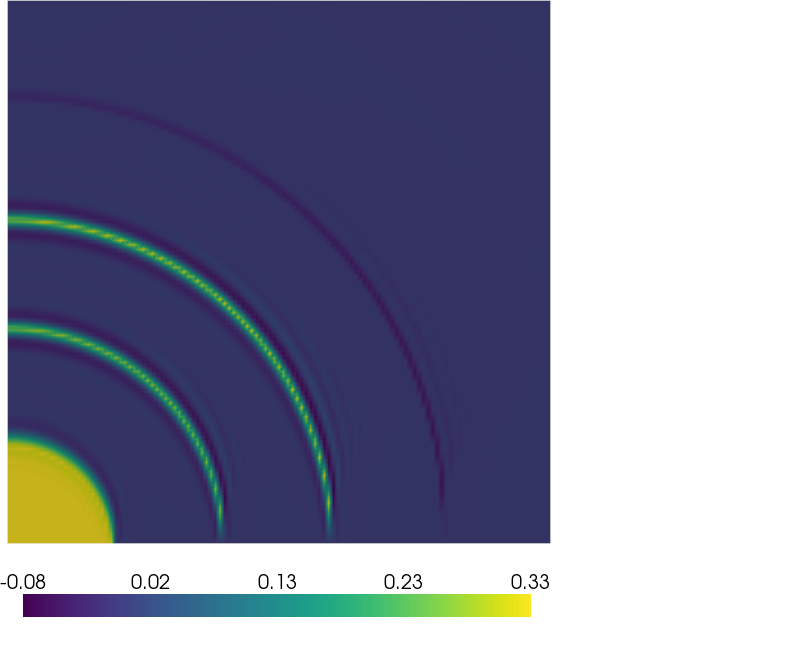}
	}
	\subfloat[Maximal eigenvalue for the CIP--CN solution.]{
		\includegraphics[width=0.45\textwidth]{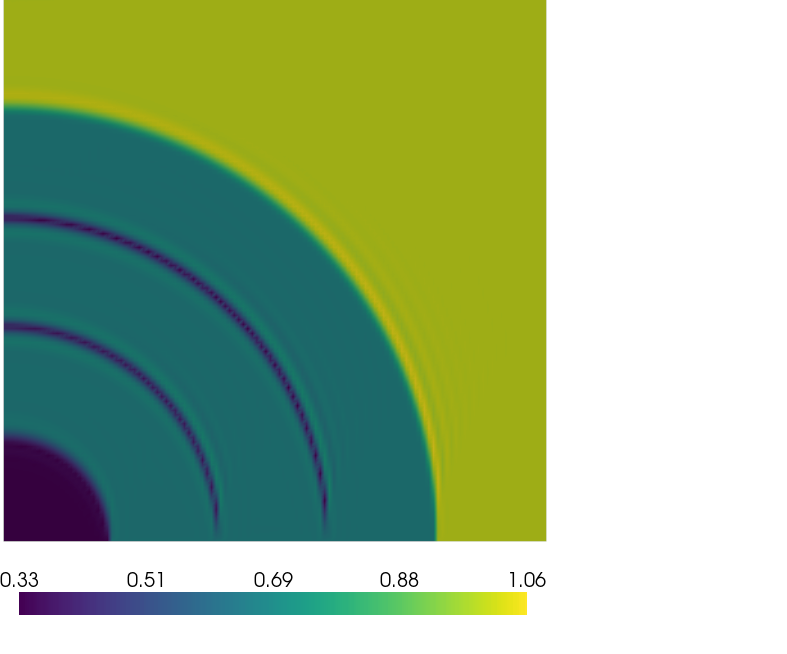}
	}\\
	\subfloat[Minimal eigenvalue for the BP--CN solution.]{
		\includegraphics[width=0.45\textwidth]{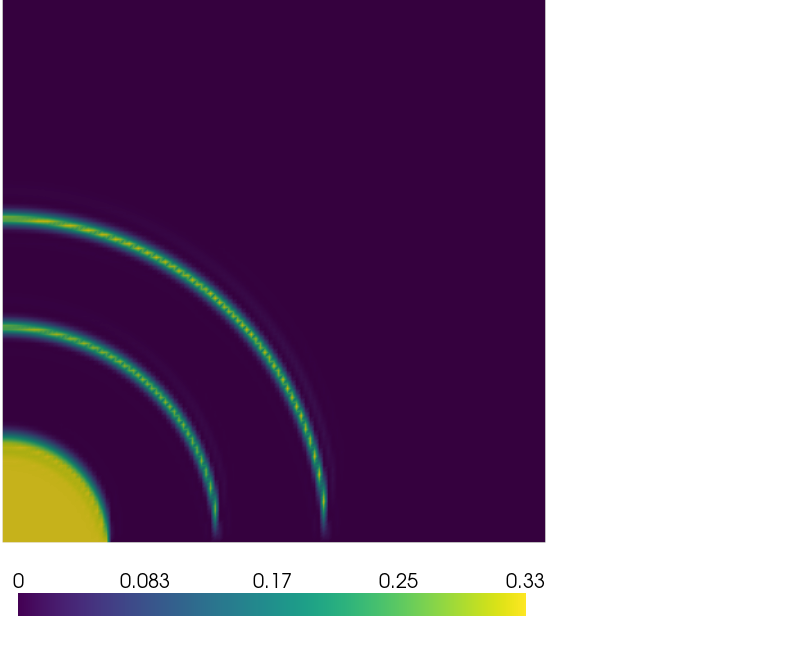}
	}
	\subfloat[Maximal eigenvalue for the BP--CN solution.]{
		\includegraphics[width=0.45\textwidth]{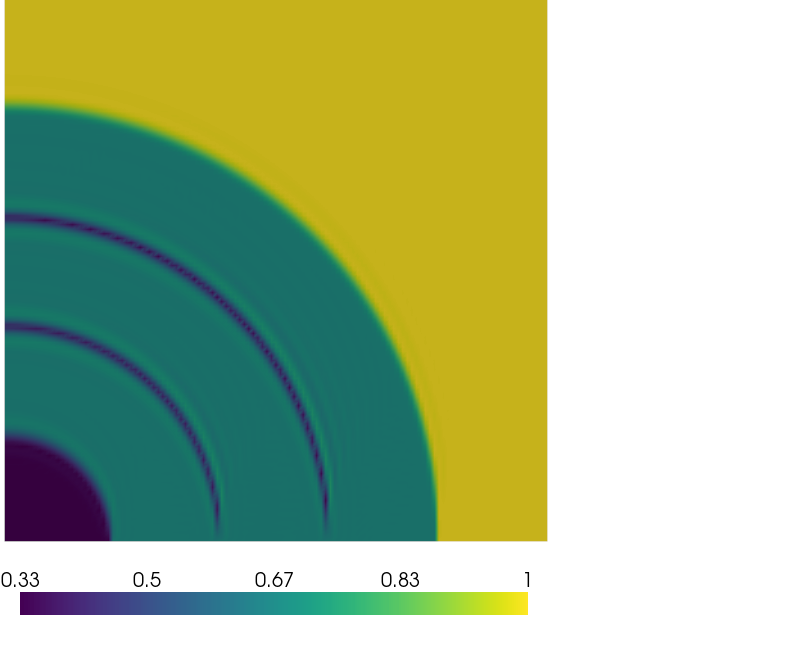}
	}\\
	\subfloat[Cross-section of the minimal eigenvalue along $y=x$.]{
		\includegraphics[width=0.45\textwidth]{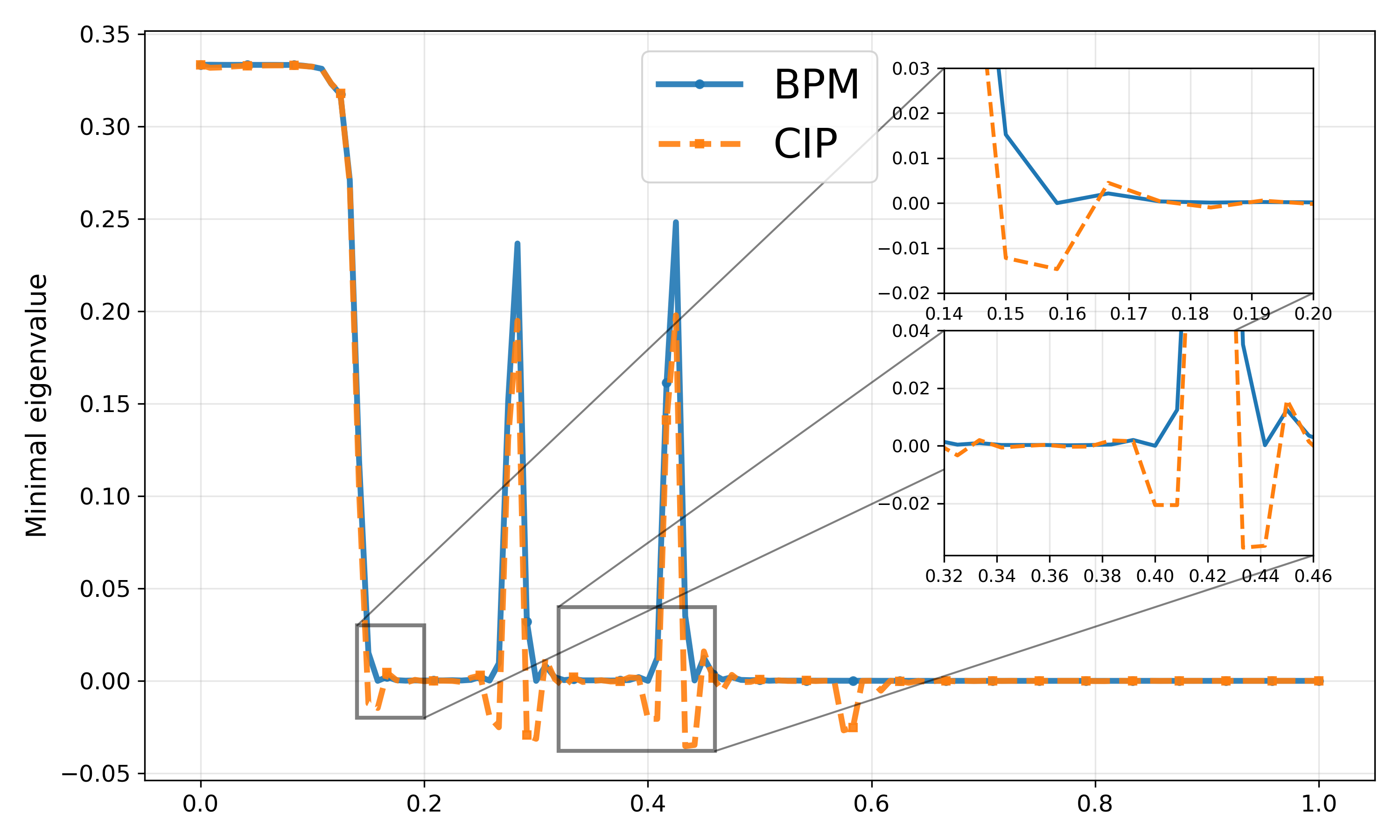}
	}
	\subfloat[Cross-section of the maximal eigenvalue along $y=x$.]{
		\includegraphics[width=0.45\textwidth]{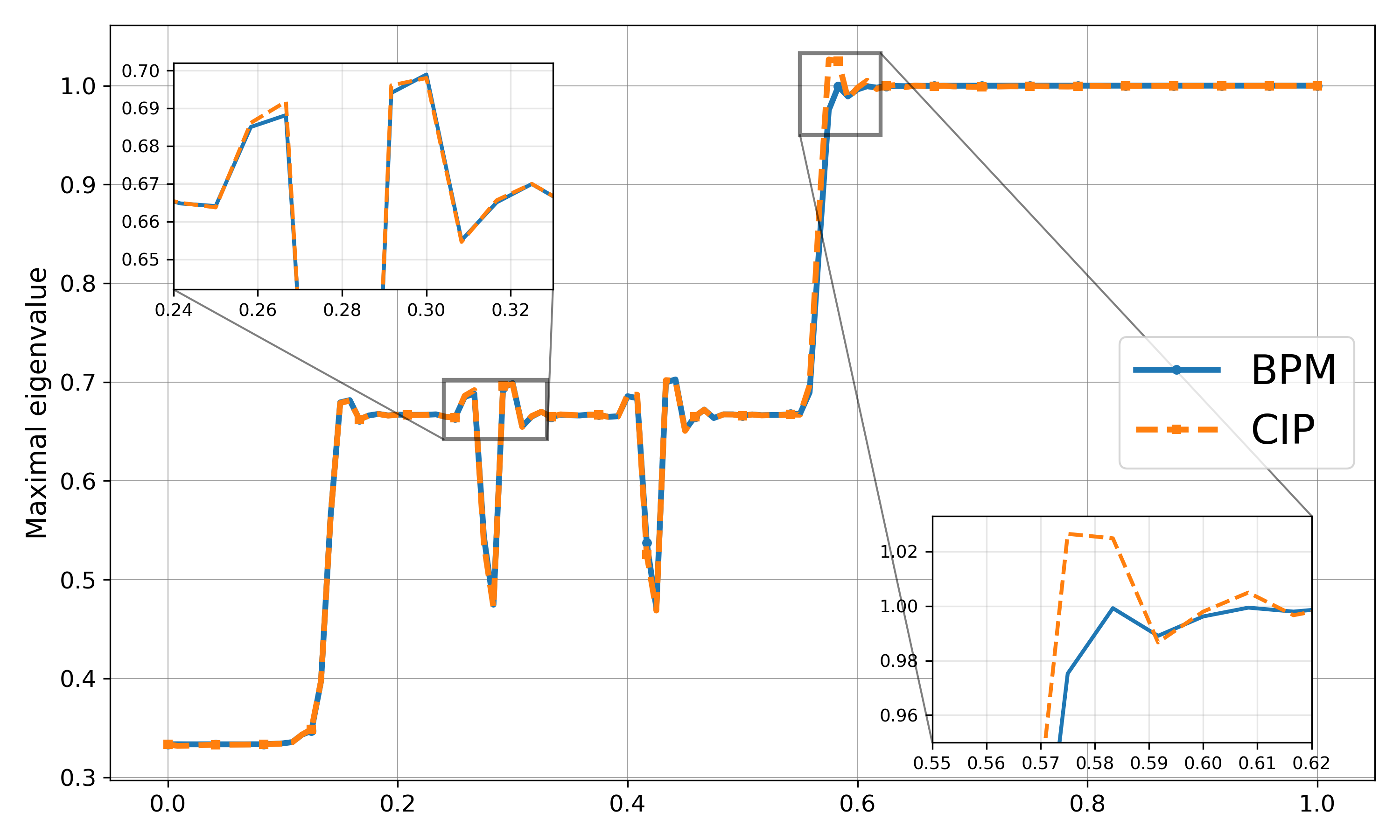}
	}\\
	\caption{Solution of \eqref{example22} at $T=4$ with discontinuous inflow data: minimal and maximal eigenvalues obtained with the CIP--CN and BP--CN schemes, together with cross-sections taken along $y=x$ ($\mathbb{P}_{1}$ elements).}
	\label{fig22}
\end{figure}

\begin{figure}[H]
	\centering
	\subfloat[Minimal eigenvalue for the CIP--Euler solution.]{
		\includegraphics[width=0.45\textwidth]{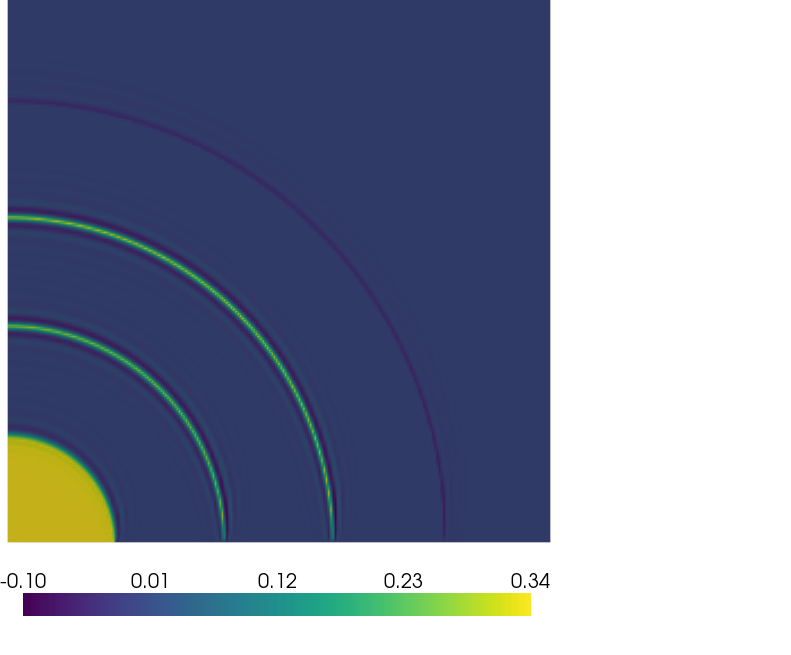}
	}
	\subfloat[Maximal eigenvalue for the CIP--Euler solution.]{
		\includegraphics[width=0.45\textwidth]{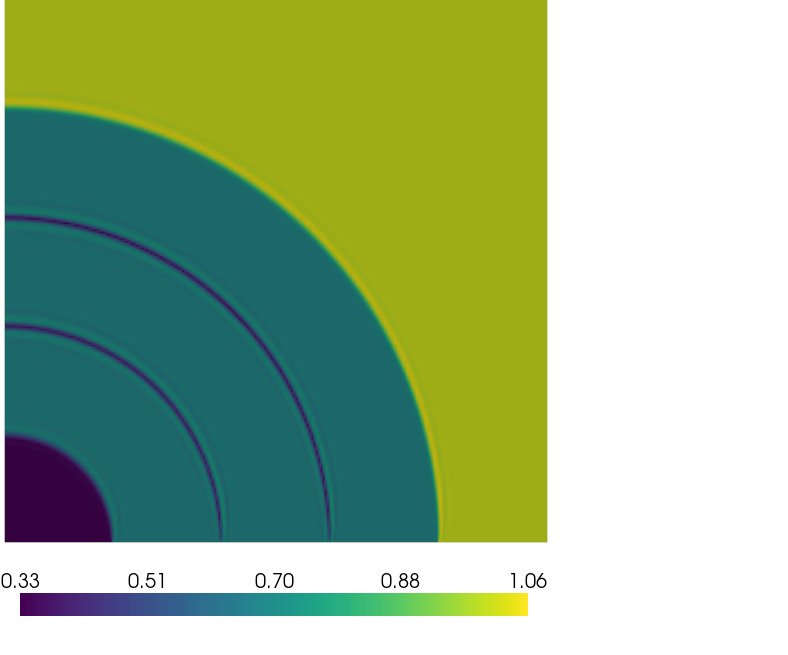}
	}\\
	\subfloat[Minimal eigenvalue for the BP--Euler solution.]{
		\includegraphics[width=0.45\textwidth]{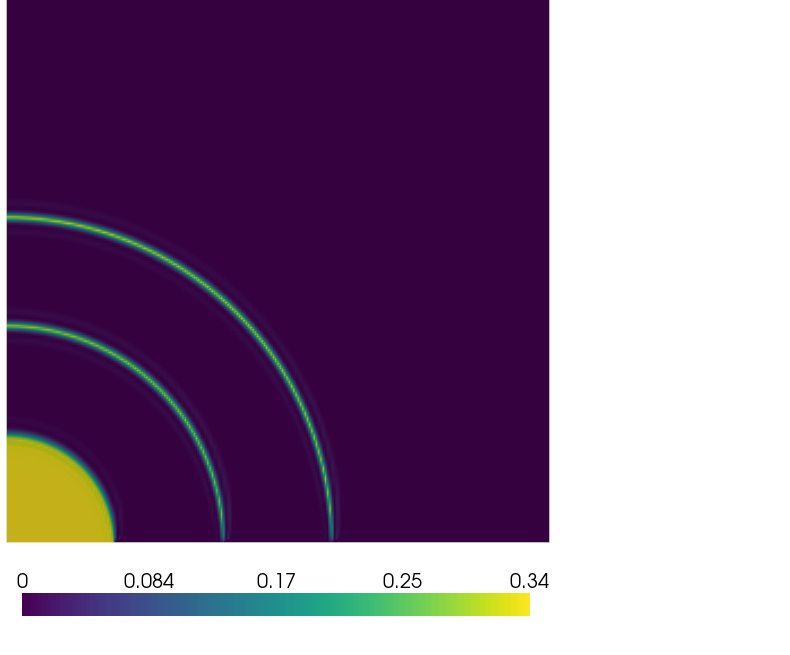}
	}
	\subfloat[Maximal eigenvalue for the BP--Euler solution.]{
		\includegraphics[width=0.45\textwidth]{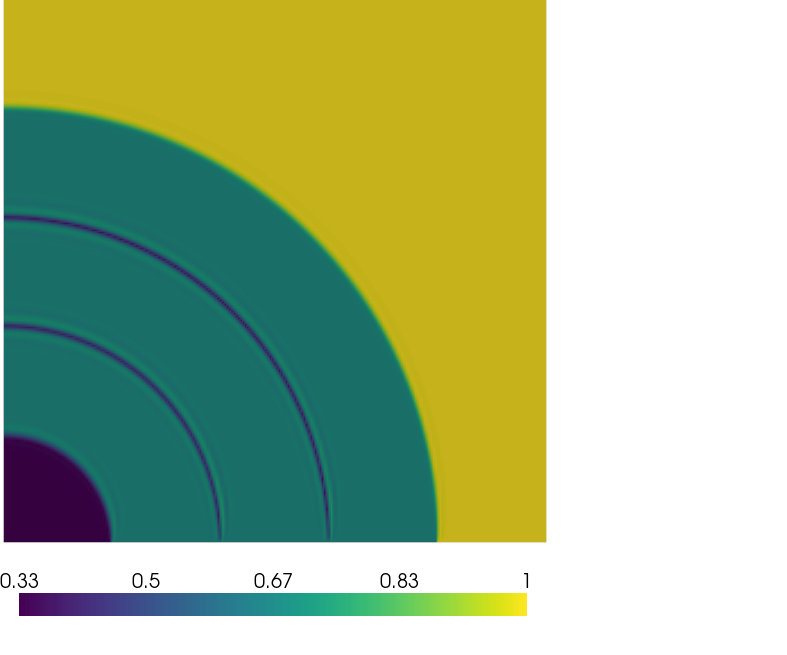}
	}\\
	\subfloat[Cross-section of the minimal eigenvalue along $y=x$.]{
		\includegraphics[width=0.45\textwidth]{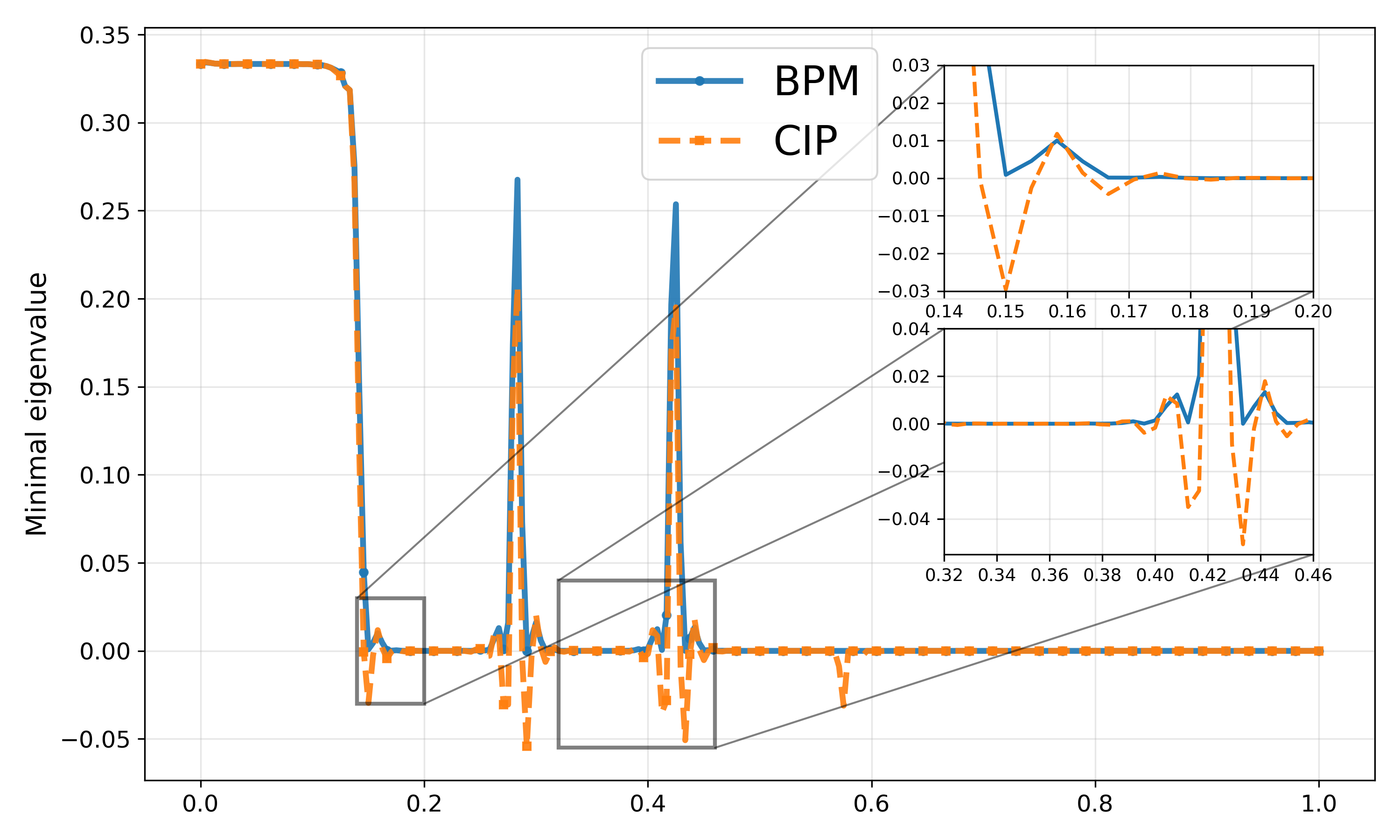}
	}
	\subfloat[Cross-section of the maximal eigenvalue along $y=x$.]{
		\includegraphics[width=0.45\textwidth]{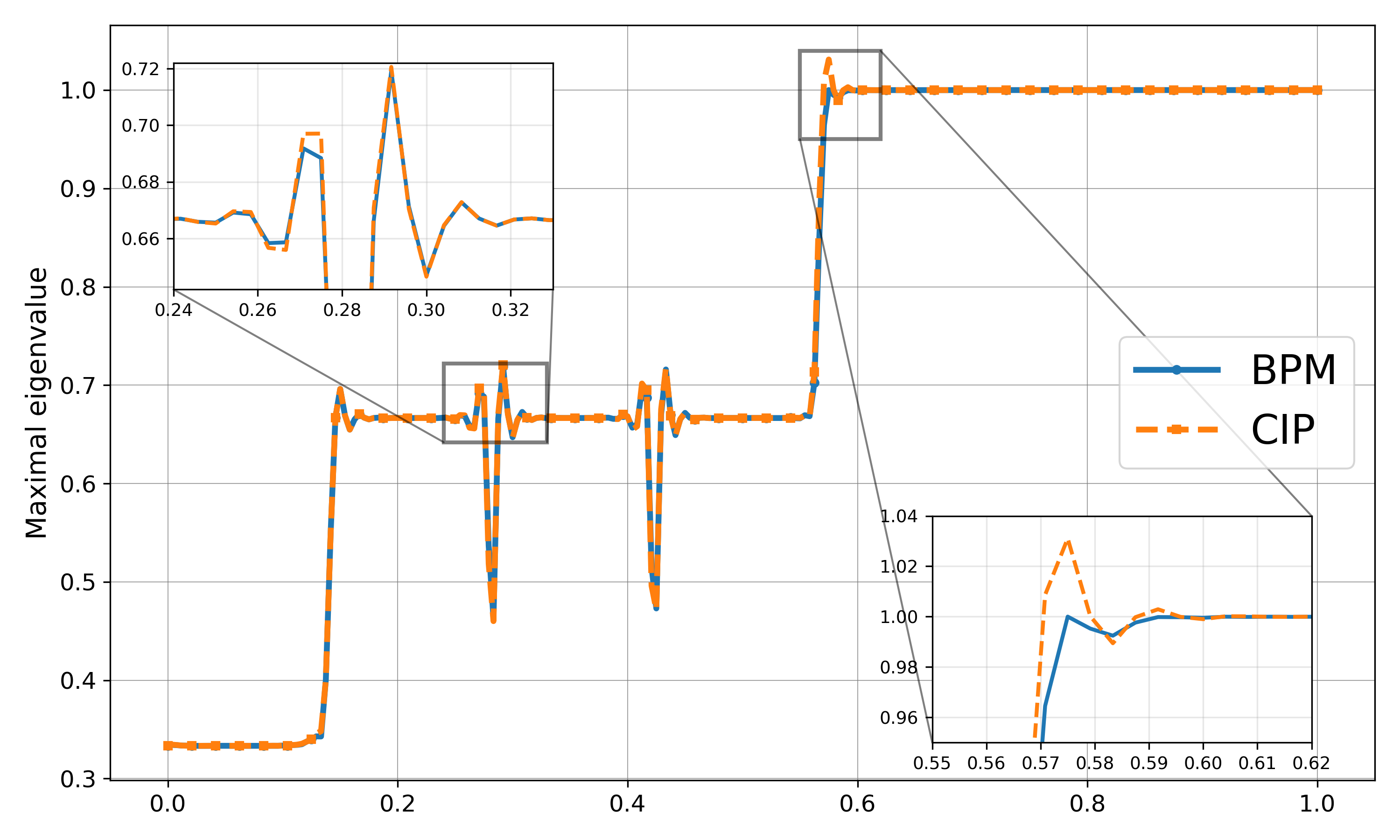}
	}\\
	\caption{Solution of \eqref{example22} at $T=4$ with discontinuous inflow data: minimal and maximal eigenvalues obtained with the CIP--Euler and BP--Euler schemes, together with cross-sections taken along $y=x$ ($\mathbb{P}_{2}$ elements).}
	\label{fig33}
\end{figure}

\begin{example}{\bf (Solid body rotation):}\label{example2}
We then test the method in a modified version of the solid body rotation benchmark originally proposed in~\cite{zalesak1979fully} and later extended in~\cite{leveque1996high}.
In~\cite{lohmann2017flux}, this benchmark was further extended to the case where the unknown is a $3\times 3$ tensor defined on the two-dimensional domain
$\Omega = (0,1)^2$ under the stationary, divergence-free velocity field
\begin{equation}
	\bbeta = \left( \tfrac{1}{2} - y, x - \tfrac{1}{2} \right)^T.
\end{equation}
Since for this test we consider $\mathcal{D}=\mathbf{0}$ and $\mu=0$,
after one full rotation, corresponding to $T = 2\pi$, the exact solution returns to the initial condition. Thus, the quality of the numerical solution is
assessed by comparing the numerical solution at $t=T$ with the initial data.

The initial condition is defined as
\begin{equation}\label{initialdata}
	\mathbf{U}_0(x,y) := \begin{cases}
		\mathbf{U}^{(1)} & \text{if } \sqrt{(x - 0.25)^2 + (y - 0.5)^2} \leq 0.15, \quad \text{`hump'}, \\
		\mathbf{U}^{(2)} & \text{if } \sqrt{(x - 0.5)^2 + (y - 0.25)^2} \leq 0.15, \quad \text{`cone'}, \\
		\mathbf{U}^{(3)} & \text{if } \sqrt{(x - 0.75)^2 + (y - 0.5)^2} \leq 0.15, \quad \text{`semi-ellipse'}, \\
		\mathbf{U}^{(4)} & \text{if } \sqrt{(x - 0.5)^2 + (y - 0.75)^2} \leq 0.15, \quad \text{`slotted cylinder'}, \\
		\mathbf{0} & \text{otherwise},
	\end{cases}
\end{equation}
which consists of four functions defined on circles centred at different points.
In what follows we use the shorthand $r := \sqrt{x^{2}+y^{2}}$.
The positive semidefinite tensors
$\mathbf{U}^{(1)}$, $\mathbf{U}^{(2)}$, $\mathbf{U}^{(3)}$ and $\mathbf{U}^{(4)}$
are specified through their eigenvalue decompositions as follows:
\[
\mathbf{U}^{(1)}(x,y) :=
\begin{pmatrix}
	1 & 0 & 0 \\
	0 & \cos\phi & \sin\phi \\
	0 & \sin\phi & -\cos\phi
\end{pmatrix}
\frac{1}{r}
\begin{pmatrix}
	x & y & 0 \\
	y & -x & 0 \\
	0 & 0 & r
\end{pmatrix}
\begin{pmatrix}
	u_1^{(1)} & 0 & 0 \\
	0 & u_2^{(1)} & 0 \\
	0 & 0 & u_3^{(1)}
\end{pmatrix}
\frac{1}{r}
\begin{pmatrix}
	x & y & 0 \\
	y & -x & 0 \\
	0 & 0 & r
\end{pmatrix}
\begin{pmatrix}
	1 & 0 & 0 \\
	0 & \cos\phi & \sin\phi \\
	0 & \sin\phi & -\cos\phi
\end{pmatrix},
\]
where
\[
u_1^{(1)} = \left( \frac{1}{2} (1+\cos(\pi r)) \right)^3, \quad
u_2^{(1)} = \left( \frac{1}{2} (1+\cos(\pi r)) \right)^2, \quad
u_3^{(1)} = \frac{1}{2} (1+\cos(\pi r)), \quad
\phi = \frac{1}{2} \operatorname{atan2}(x,y),
\]
\[
\mathbf{U}^{(2)}(x,y) := \frac{1}{10}
\begin{pmatrix}
	10 & 0 & 0 \\
	0 & 8 & 6 \\
	0 & 6 & -8
\end{pmatrix}
\frac{1}{r}
\begin{pmatrix}
	x & y & 0 \\
	y & -x & 0 \\
	0 & 0 & r
\end{pmatrix}
\begin{pmatrix}
	u_1^{(2)} & 0 & 0 \\
	0 & u_a^{(2)} & 0 \\
	0 & 0 & u_b^{(2)}
\end{pmatrix}
\frac{1}{r}
\begin{pmatrix}
	x & y & 0 \\
	y & -x & 0 \\
	0 & 0 & r
\end{pmatrix}
\frac{1}{10}
\begin{pmatrix}
	10 & 0 & 0 \\
	0 & 8 & 6 \\
	0 & 6 & -8
\end{pmatrix},
\]
where $u^{(2)}_1 =\frac{1}{2}- \frac{1}{2} r$, $u^{(2)}_a = \frac{1}{2} - \frac{1}{2} |x|$, and $u^{(2)}_b = 1 - r$,
\begin{equation}
	\mathbf{U}^{(3)}(x,y) := \begin{pmatrix} u^{(3)}_1 & 0 & 0 \\ 0 & u^{(3)}_1 & 0 \\ 0 & 0 & u^{(3)}_1 \end{pmatrix},
\end{equation}
where
\[
u_1^{(3)} = u_2^{(3)} = u_3^{(3)} = \sqrt{1 - r^2},
\]
\[
\mathbf{U}^{(4)}(x,y) :=
\begin{cases}
	\frac{1}{10}
	\begin{pmatrix}
		-8 & 6 & 0 \\
		6 & 8 & 0 \\
		0 & 0 & 10
	\end{pmatrix}
	\begin{pmatrix}
		u_3^{(4)} & 0 & 0 \\
		0 & u_1^{(4)} & 0 \\
		0 & 0 & u_2^{(4)}
	\end{pmatrix}
	\frac{1}{10}
	\begin{pmatrix}
		-8 & 6 & 0 \\
		6 & 8 & 0 \\
		0 & 0 & 10
	\end{pmatrix}, & (|x| \geq \frac{1}{6} \vee y \geq \frac{2}{3}) \wedge (x > 0), \\[0.2cm]
	\frac{1}{10}
	\begin{pmatrix}
		-8 & 6 & 0 \\
		6 & 8 & 0 \\
		0 & 0 & 10
	\end{pmatrix}
	\begin{pmatrix}
		u_3^{(4)} & 0 & 0 \\
		0 & u_3^{(4)} & 0 \\
		0 & 0 & u_2^{(4)}
	\end{pmatrix}
	\frac{1}{10}
	\begin{pmatrix}
		-8 & 6 & 0 \\
		6 & 8 & 0 \\
		0 & 0 & 10
	\end{pmatrix}, & (|x| \geq \frac{1}{6} \vee y \geq \frac{2}{3}) \wedge (x < 0), \\[0.2cm]
	\mathbf{0}, & \text{elsewhere},
\end{cases}
\]
where
\[
u_1^{(4)} = 0.1, \quad u_2^{(4)} = 0.45, \quad u_3^{(4)} = 1.
\]

The intermediate and largest eigenvalues of $\mathbf{U}^{(2)}$ are given by
\[
u_2^{(2)} =
\begin{cases}
	u_a^{(2)}, & |x| \geq 2r - 1, \\
	u_b^{(2)}, & |x| < 2r - 1,
\end{cases}
\hspace{1cm}
u_3^{(2)} =
\begin{cases}
	u_b^{(2)}, & |x| \geq 2r - 1, \\
	u_a^{(2)}, & |x| < 2r - 1.
\end{cases}
\]
The minimal and maximal eigenvalues of $\mathbf{U}_0$ are $0$ and $1$, respectively. Thus, we take $\epsilon=0$ and $\kappa=1$.

The labels \emph{hump}, \emph{cone}, \emph{semi-ellipse} and \emph{slotted cylinder} correspond to the design of the respective minimal and maximal eigenvalues,
reflecting their similarity to the scalar solid-body rotation benchmark.
The bodies undergo a counter-clockwise rotation, completing one full revolution at $t = 2\pi$.
For the space discretisation in this example, we employ piecewise linear elements (that is, $k=1$).
We set $P = 121$, which yields a uniform grid with $2\times 121\times 121$ mesh cells,
and use a time step of $\Delta t = 5\times 10^{-4}$. In all experiments in this example we use $\omega=0.01$ in the iterative method \eqref{iterative12} and $\gamma=10^{-3}$ in the stabilisation term \eqref{eq10}.

Figure~\ref{fig11111} illustrates the minimal and maximal eigenvalues of the initial data \eqref{initialdata}.
Figure~\ref{fig111} presents the minimal and maximal eigenvalues of the numerical solution obtained by BP--Euler and CIP--Euler at the final time $t=T$.
To compare the performance of these two approaches, we consider a cross-section along the line $y=0.8$ extracted from both the minimal and maximal eigenvalues of the initial data \eqref{initialdata} and from the corresponding numerical solutions.
This comparison enables us to assess how effectively the bound-preserving method prevents overshoots and undershoots compared to the CIP method.
From the eigenvalue plots and cross-sections, it is evident that the BP method successfully enforces the lower and upper bounds on the eigenvalues of the solution.
In the cross-section plots, \emph{initial} denotes the eigenvalues of the initial datum.

Figure~\ref{fig2} presents analogous results for the BP--CN and CIP--CN schemes.
While BP--CN ensures that the solution remains within the physical bounds, CIP--CN may exhibit violations of these bounds, as is evident in both the eigenvalue plots and the cross-sections at $y=0.8$.

\end{example}

\begin{figure}[H]
\centering
\subfloat[Minimal eigenvalue of the initial datum.]{
	\includegraphics[width=0.45\textwidth]{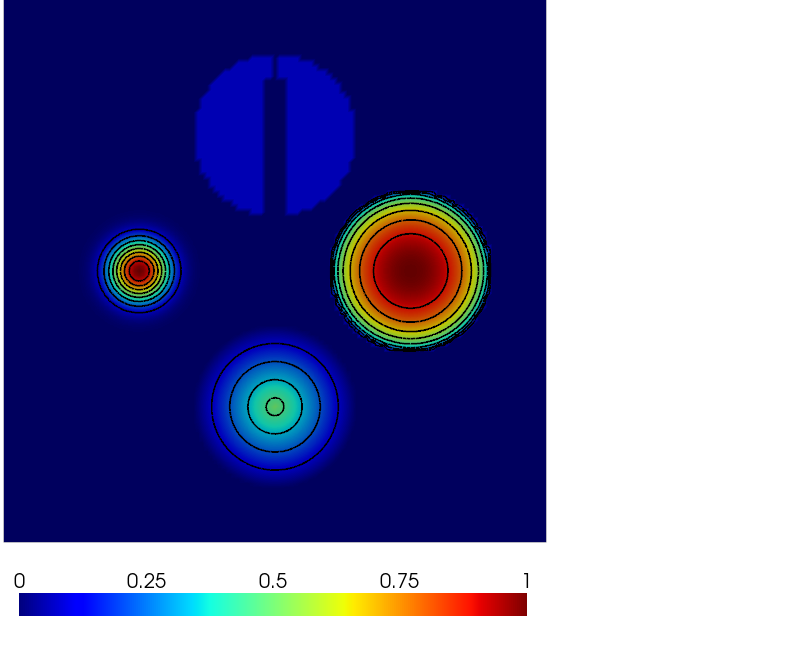}
}
\subfloat[Maximal eigenvalue of the initial datum.]{
	\includegraphics[width=0.45\textwidth]{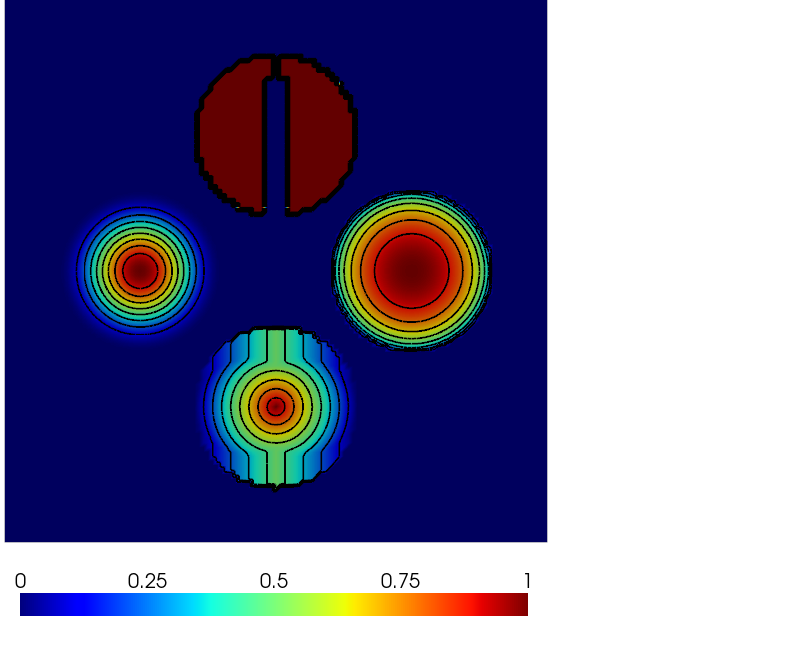}
}
\caption{Minimal and maximal eigenvalues of the initial datum \eqref{initialdata}.}
\label{fig11111}
\end{figure}

\begin{figure}[H]
	\centering
	\subfloat[Minimal eigenvalue for the CIP--Euler solution.]{
		\includegraphics[width=0.45\textwidth]{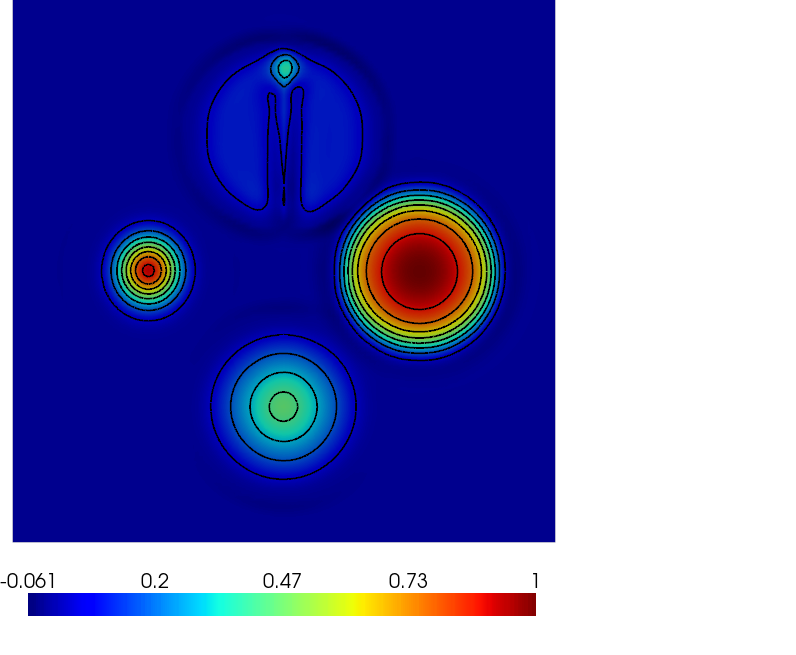}
	}
	\subfloat[Maximal eigenvalue for the CIP--Euler solution.]{
		\includegraphics[width=0.45\textwidth]{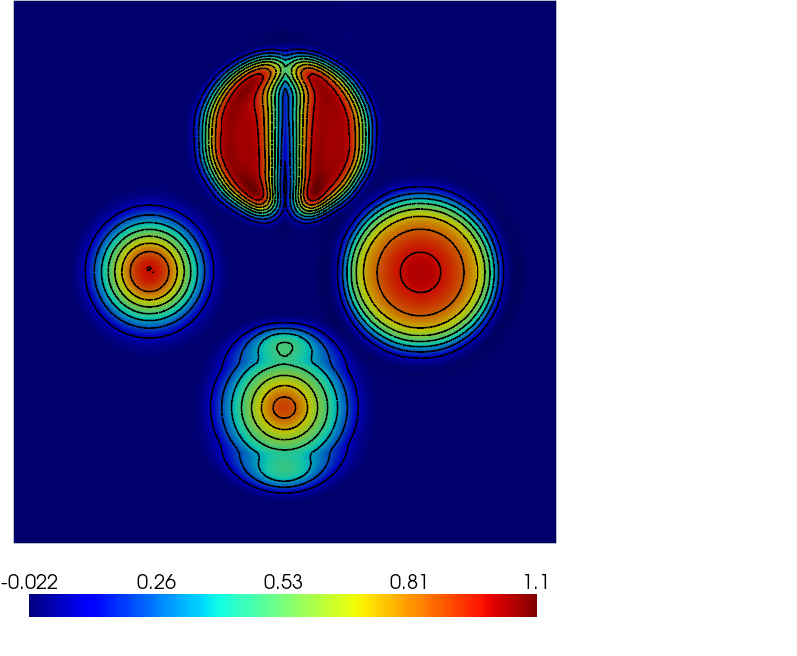}
	}\\
	\subfloat[Minimal eigenvalue for the BP--Euler solution.]{
		\includegraphics[width=0.45\textwidth]{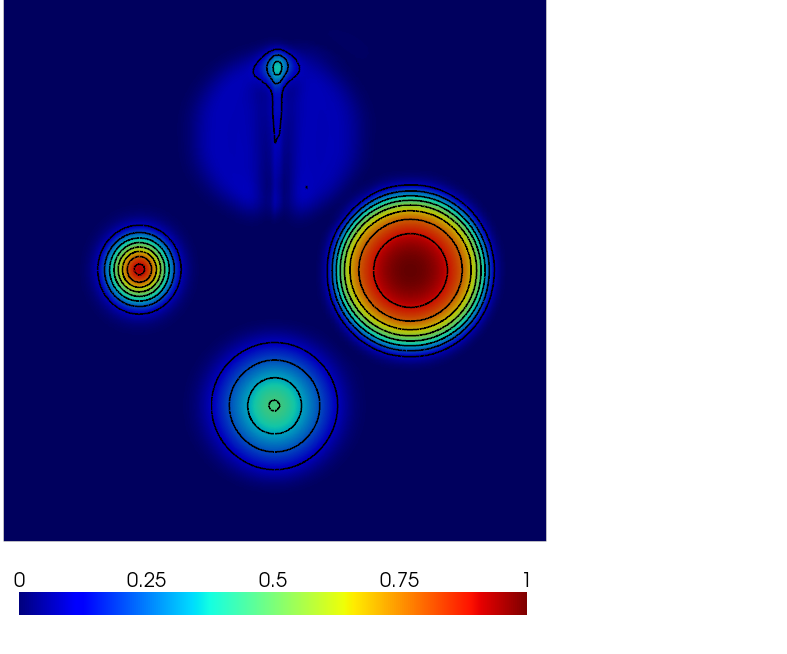}
	}
	\subfloat[Maximal eigenvalue for the BP--Euler solution.]{
		\includegraphics[width=0.45\textwidth]{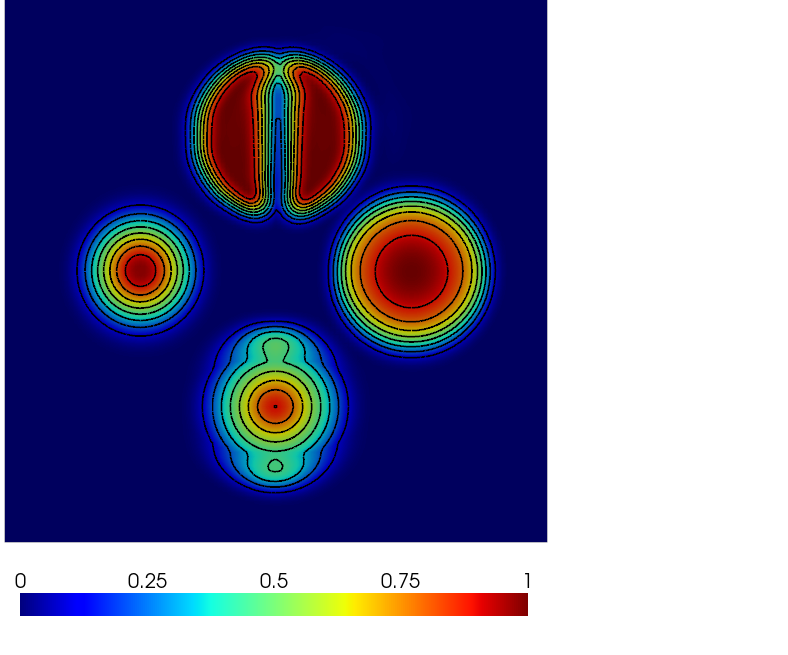}
	}\\
	\subfloat[Cross-section of the minimal eigenvalue along $y=0.8$.]{
		\includegraphics[width=0.45\textwidth]{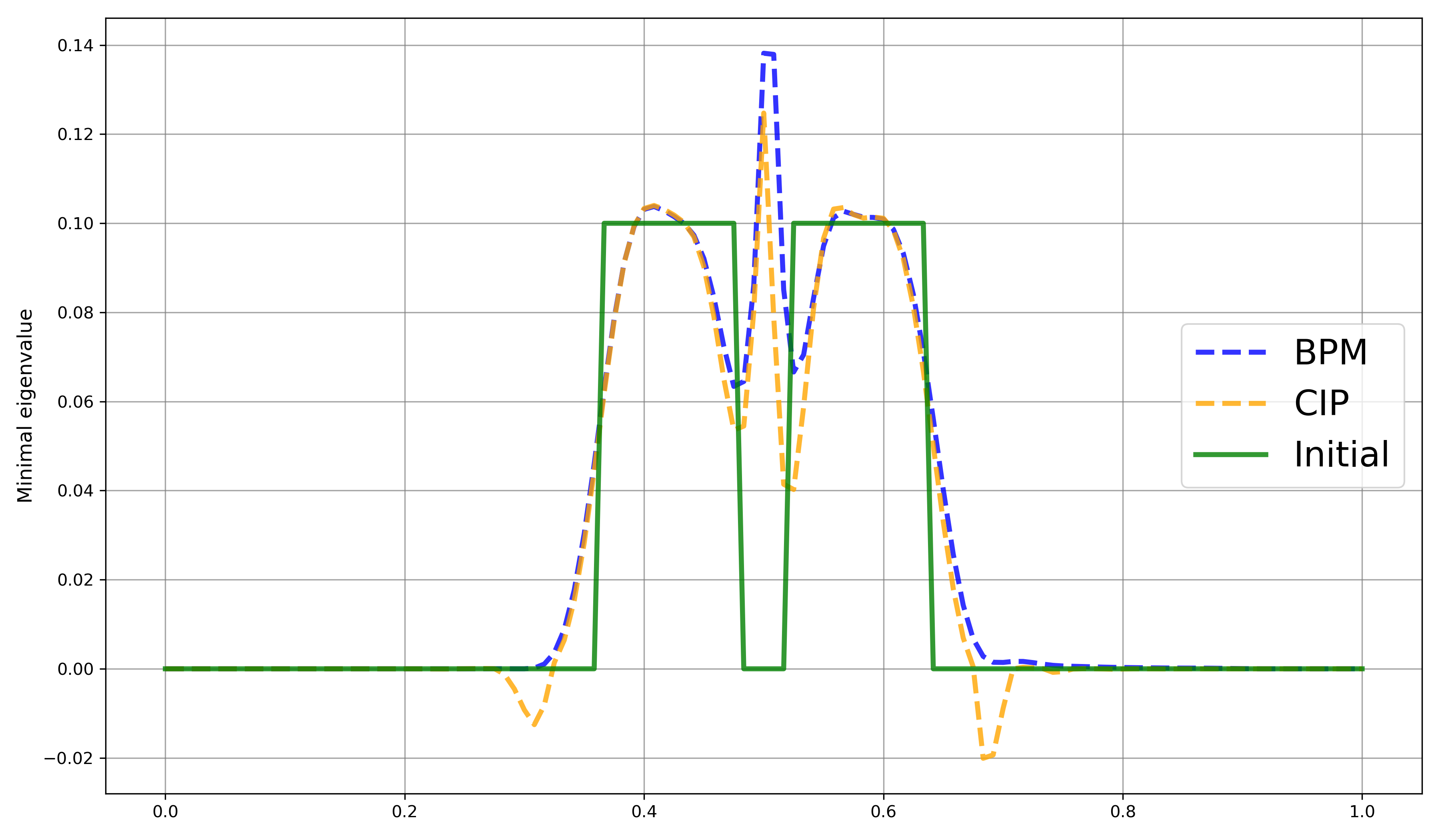}
	}
	\subfloat[Cross-section of the maximal eigenvalue along $y=0.8$.]{
		\includegraphics[width=0.45\textwidth]{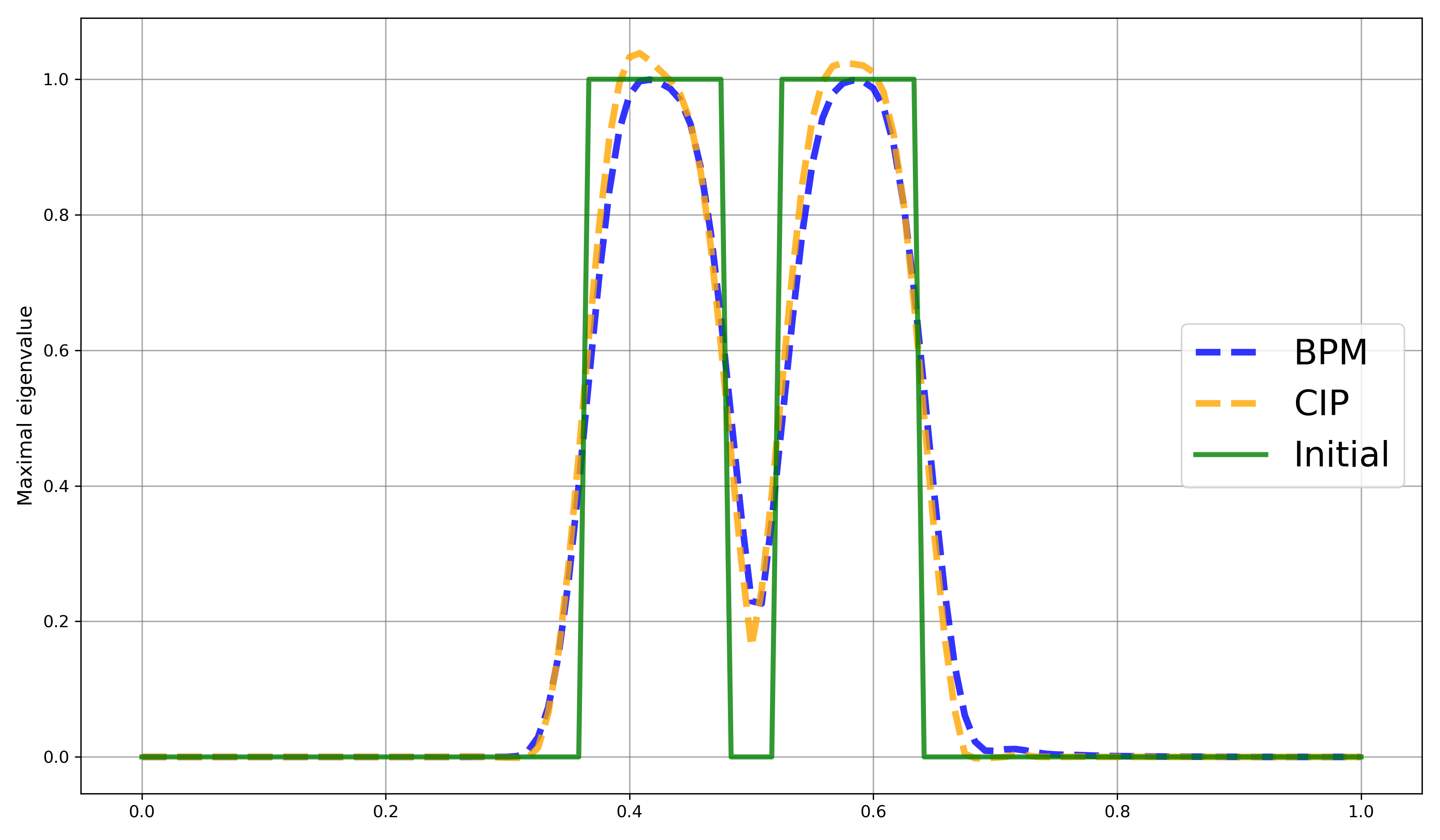}
	}
	\caption{Example~\ref{example2} at $t=T=2\pi$: minimal and maximal eigenvalues obtained with the CIP--Euler and BP--Euler schemes, together with cross-sections along $y=0.8$.}
	\label{fig111}
\end{figure}

\begin{figure}[H]
	\centering
	\subfloat[Minimal eigenvalue for the CIP--CN solution.]{
		\includegraphics[width=0.45\textwidth]{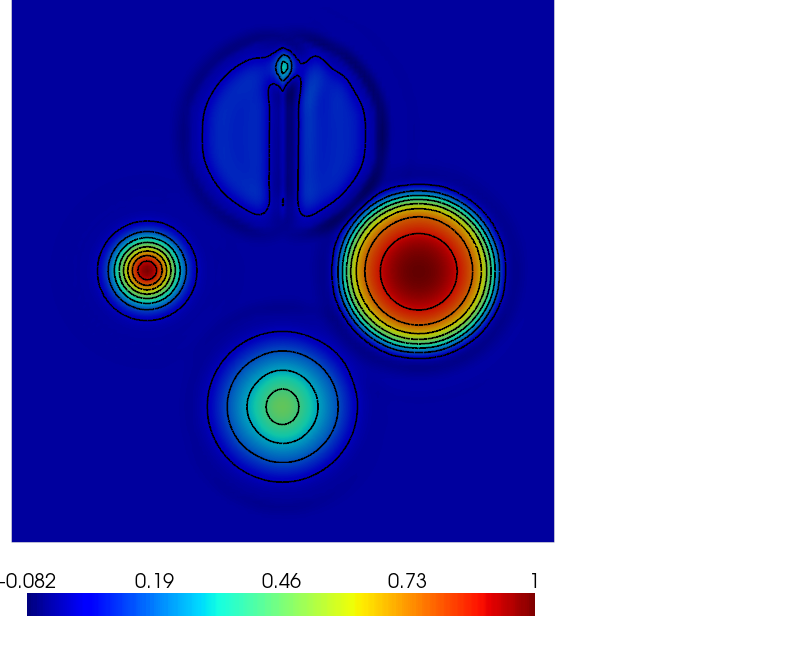}
	}
	\subfloat[Maximal eigenvalue for the CIP--CN solution.]{
		\includegraphics[width=0.45\textwidth]{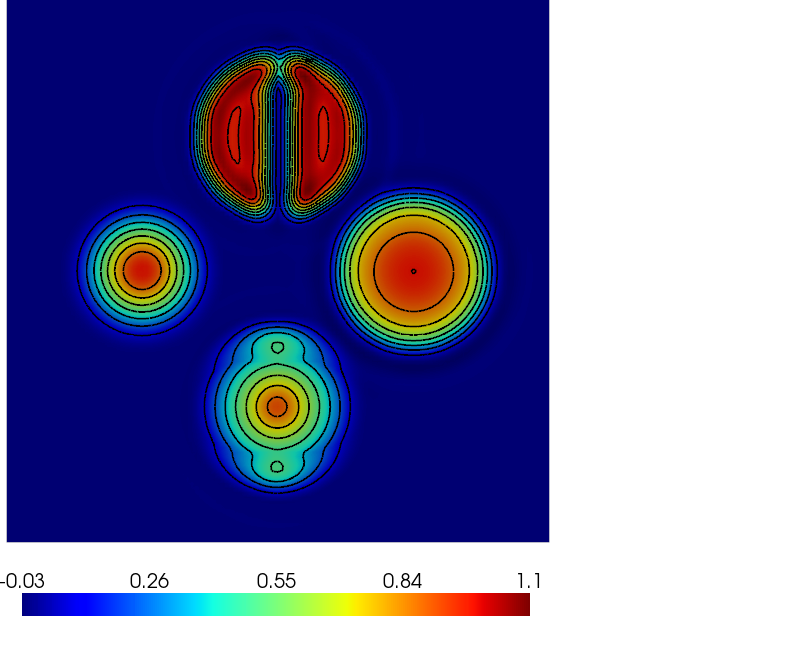}
	}\\
	\subfloat[Minimal eigenvalue for the BP--CN solution.]{
		\includegraphics[width=0.45\textwidth]{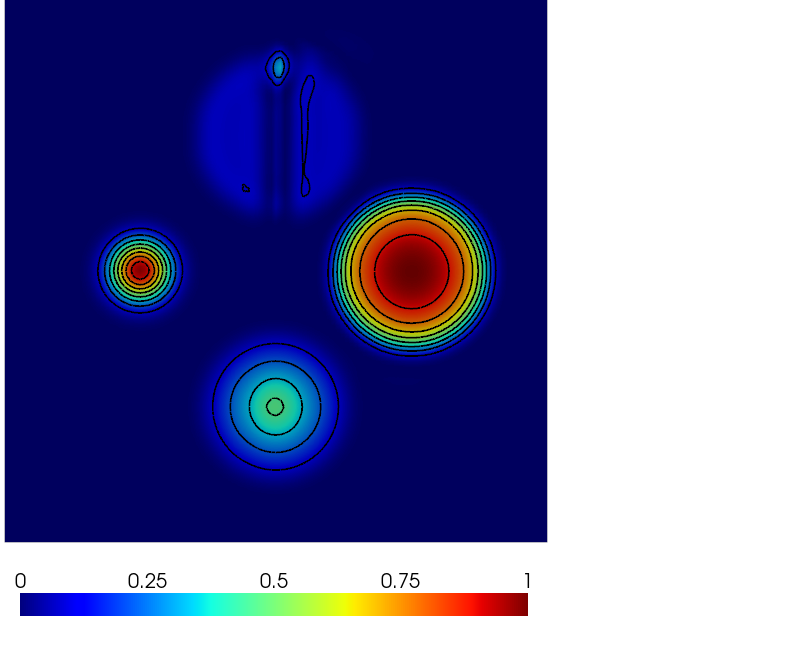}
	}
	\subfloat[Maximal eigenvalue for the BP--CN solution.]{
		\includegraphics[width=0.45\textwidth]{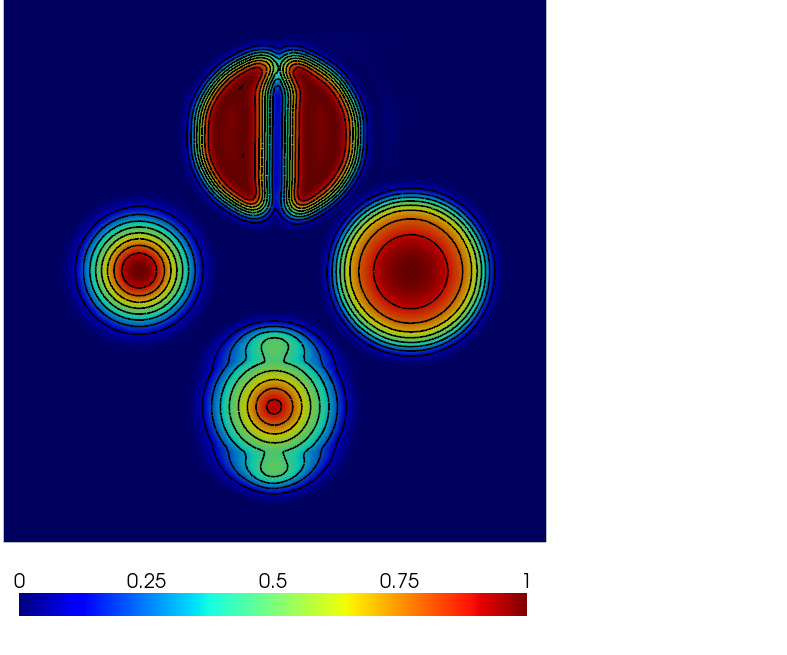}
	}\\
	\subfloat[Cross-section of the minimal eigenvalue along $y=0.8$.]{
		\includegraphics[width=0.45\textwidth]{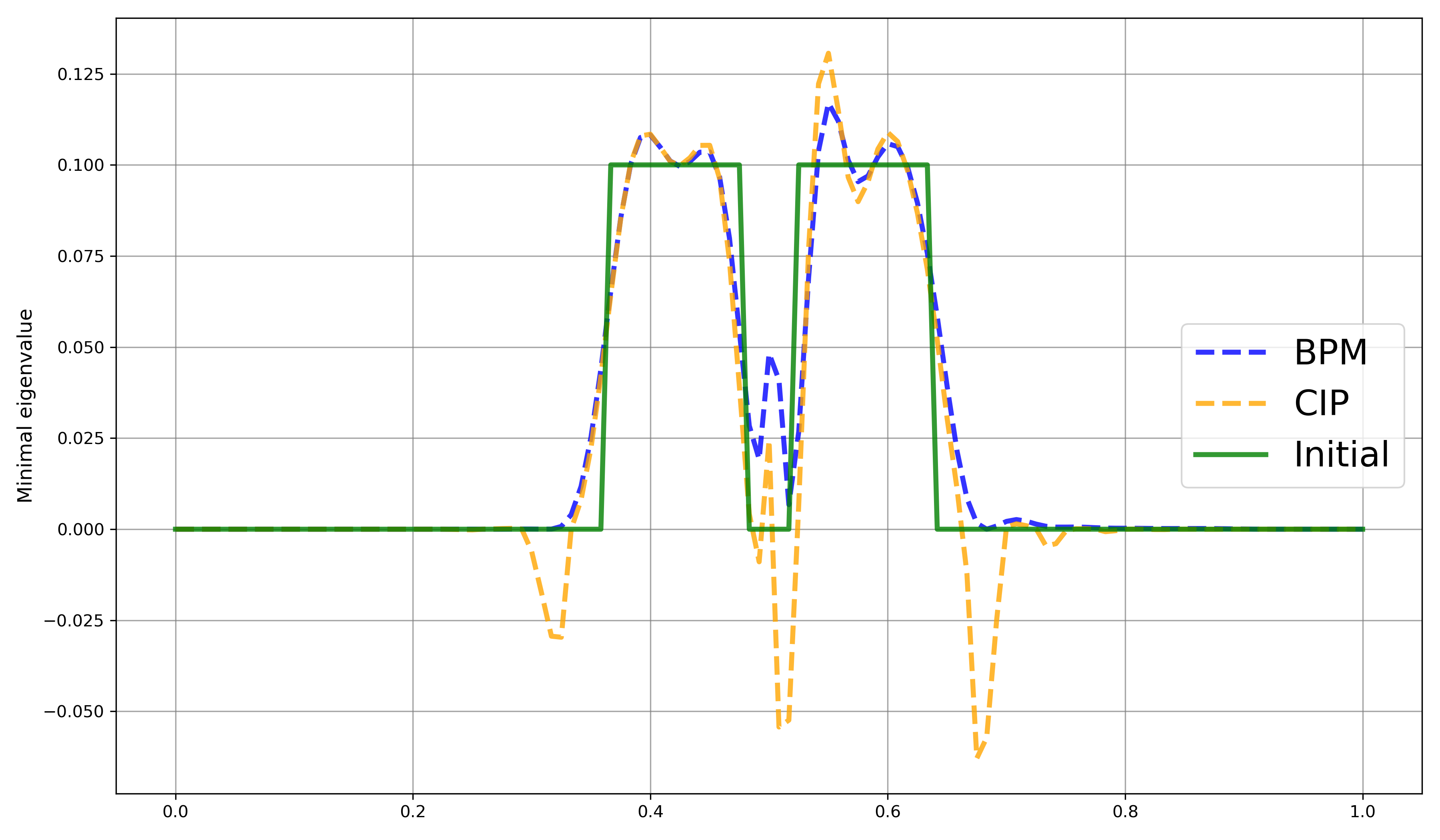}
	}
	\subfloat[Cross-section of the maximal eigenvalue along $y=0.8$.]{
		\includegraphics[width=0.45\textwidth]{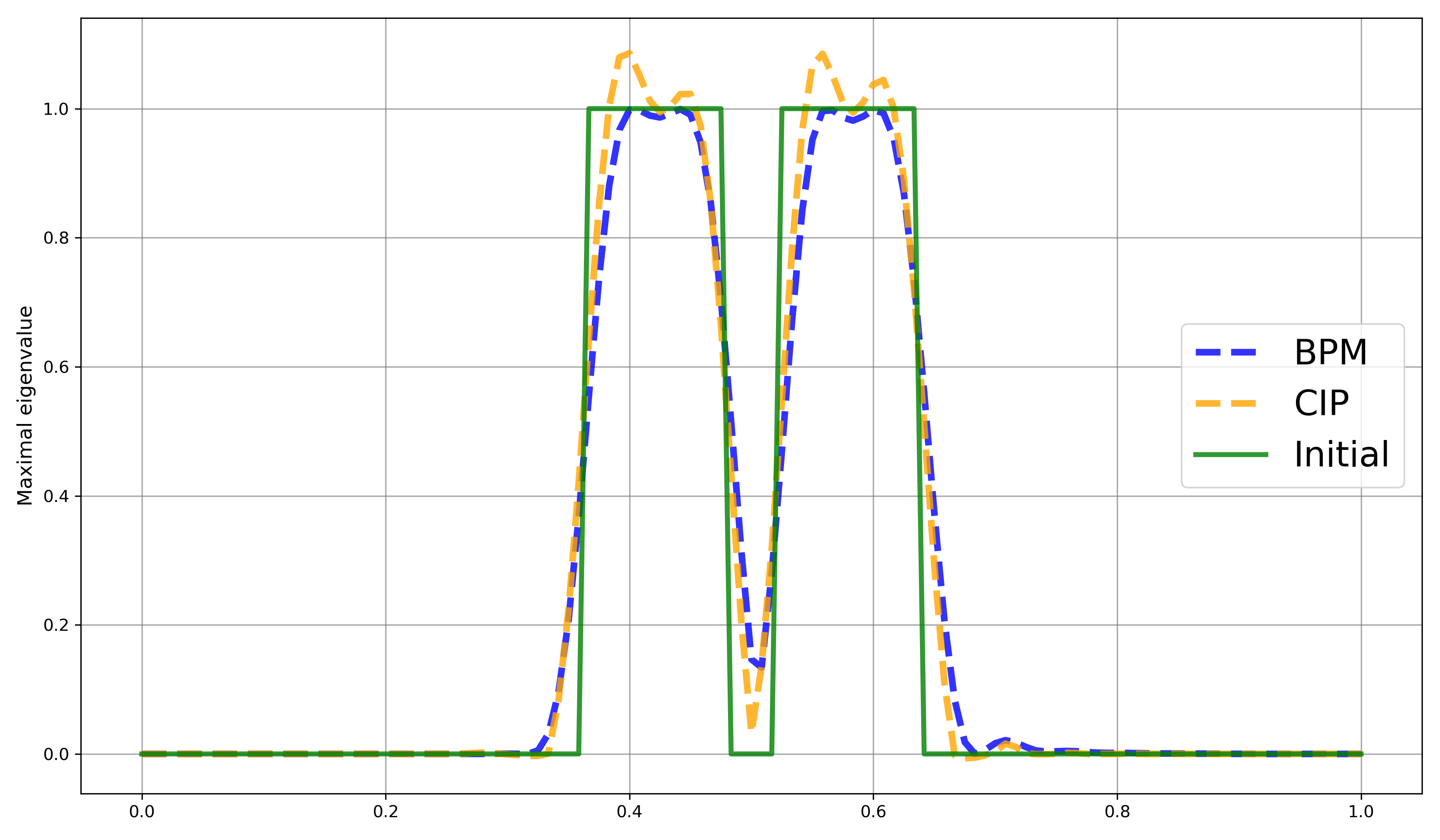}
	}
	\caption{Example~\ref{example2} at $t=T=2\pi$: minimal and maximal eigenvalues obtained with the CIP--CN and BP--CN schemes, together with cross-sections along $y=0.8$.}
	\label{fig2}
\end{figure}

\section{Conclusions and outlook}
\label{sec:conc}

In this work we have extended the nodally bound-preserving idea to the
numerical approximation of time-dependent tensor-valued partial differential equations. 
The method is built using a variational inequality posed over a set of acceptable
functions, which is proven to be well-posed, and yielding to an unconditionally
stable and optimally-convergent discretisation. 
Our approach offers several advantages over existing methods. Unlike
post-processing techniques, our method directly enforces the
eigenvalue constraints as part of the variational formulation,
ensuring that the discrete solution satisfies the physical constraints
at the degrees of freedom at each time step. In contrast
to transformation-based approaches, our method maintains a linear
structure for linear problems in the unconstrained limit,
  and in the sense that the baseline discretisation remains linear,
simplifying both the analysis and implementation. Moreover, the CIP
stabilisation provides robustness in the convection-dominated regime
without compromising accuracy in diffusion-dominated regions
through its consistent interior penalty formulation. 

There are still several questions that remain open.  For example, the extension of the analysis
to higher-order time discretisations,   exploring the possibility of combining the
current method with IMEX-type schemes for nonlinear problems, and whether this framework
can be extended to nonlinear conservation laws are very interesting open
questions that will be subject of future research.

\section*{Acknowledgements}
The work of AA, GRB, and TP has been partially supported by the
Leverhulme Trust Research Project Grant No. RPG-2021-238.  TP is also
partially supported by EPRSC grants
\href{https://gow.epsrc.ukri.org/NGBOViewGrant.aspx?GrantRef=EP/W026899/2}
{EP/W026899/2},
\href{https://gow.epsrc.ukri.org/NGBOViewGrant.aspx?GrantRef=EP/X017206/1}
{EP/X017206/1}
and
\href{https://gow.epsrc.ukri.org/NGBOViewGrant.aspx?GrantRef=EP/X030067/1}
{EP/X030067/1}.

\bibliographystyle{alpha}
\bibliography{refs.bib}

\end{document}